\documentclass[10pt]{article}

\usepackage{latexsym, amssymb, amsmath, amsthm, amscd}
\usepackage[all,cmtip]{xy}
\usepackage{amsmath}
\usepackage{amsthm}
\usepackage{amssymb}
\usepackage{amsfonts}
\usepackage{amsrefs}
\usepackage{color,authblk}
\usepackage{tikz}

\usepackage{amscd} 
\usepackage{comment}
\usepackage{mathrsfs}
\usepackage[all]{xy}
\usepackage{graphicx}
\usepackage{authblk}
\usepackage{hyperref}
\usepackage{fullpage}
\usepackage[all,cmtip]{xy}
\usepackage{relsize}
\usepackage{amsmath,amscd}
\usepackage{amssymb}
\usepackage{stackrel}
\usepackage{tikz-cd}
\usepackage{tikz}
\usetikzlibrary{arrows}
\usetikzlibrary{matrix}
\usepackage[mathscr]{euscript}

\usepackage[all]{xy}

\usepackage{hyperref}

\numberwithin{equation}{section} \DeclareMathSizes{2}{10}{12}{13}
\makeatletter
\newcommand*{\doublerightarrow}[2]{\mathrel{
  \settowidth{\@tempdima}{$\scriptstyle#1$}
  \settowidth{\@tempdimb}{$\scriptstyle#2$}
  \ifdim\@tempdimb>\@tempdima \@tempdima=\@tempdimb\fi
  \mathop{\vcenter{
    \offinterlineskip\ialign{\hbox to\dimexpr\@tempdima+1em{##}\cr
    \rightarrowfill\cr\noalign{\kern.5ex}
    \rightarrowfill\cr}}}\limits^{\!#1}_{\!#2}}}

\newcommand{\leftrarrows}{\mathrel{\raise.75ex\hbox{\oalign{%
  $\scriptstyle\leftarrow$\cr
  \vrule width0pt height.5ex$\hfil\scriptstyle\relbar$\cr}}}}
\newcommand{\lrightarrows}{\mathrel{\raise.75ex\hbox{\oalign{%
  $\scriptstyle\relbar$\hfil\cr
  $\scriptstyle\vrule width0pt height.5ex\smash\rightarrow$\cr}}}}
\newcommand{\Rrelbar}{\mathrel{\raise.75ex\hbox{\oalign{%
  $\scriptstyle\relbar$\cr
  \vrule width0pt height.5ex$\scriptstyle\relbar$}}}}

\makeatletter
\def\leftrightarrowsfill@{\arrowfill@\leftrarrows\Rrelbar\lrightarrows}
\newcommand{\xleftrightarrows}[2][]{\ext@arrow 3399\leftrightarrowsfill@{#1}{#2}}
\makeatother

\newtheorem{thm}{Proposition}[section]
\newtheorem{Thm}[thm]{Theorem}
\newtheorem{rem}[thm]{Remark}

\newtheorem{cor}[thm]{Corollary}

\newtheorem{lem}[thm]{Lemma}
\newtheorem{defn}[thm]{Definition}

\title{Heavily separable functors of the second kind and applications
}

\smallskip 

\smallskip 
\author{Abhishek Banerjee\footnote{Email: abhishekbanerjee1313@gmail.com} $\qquad \qquad$ Subhajit Das\footnote{Email: subhajitdas@iisc.ac.in}}
\date{ }

\begin{document}

\maketitle

\centerline{\small \emph{Department of Mathematics, Indian Institute of Science, Bangalore 560012, India.}}

\smallskip

\begin{abstract} We introduce heavily separable functors of the second kind and study them in three different situations. The first of these is with restrictions and extensions of scalars for modules over small preadditive categories. The second is with free functors taking values in Eilenberg-Moore categories associated to a monad or a comonad. Finally, we consider entwined modules and give if and only if  conditions for heavy separability of the second kind for functors forgetting either the comodule action or the module action.
\end{abstract}

\smallskip

{\emph{{\bf\emph{MSC(2020) Subject Classification:} } }}  16D90, 18A40, 18C20

\smallskip

{ \emph{{\bf \emph{Keywords:}}} separable, monads, comonads, entwined modules

\smallskip

\section{Introduction}

The purpose of this paper is to introduce the notion of heavily separable functors of the second kind and study it in three different contexts. The categorical notion of a separable functor was first given by N\u{a}st\u{a}sescu,  Van den Bergh,  and Van Oystaeyen in \cite{NVV}. A functor $F:\mathcal C\longrightarrow \mathcal D$ is said to be separable  if the natural transformation on morphism spaces induced by $F$ can be split by a natural transformation $P$. This definition is constructed so that separable morphisms of rings correspond to the restriction of scalars being separable in the sense of \cite{NVV}. Therefore, the study of separable functors is closely tied to that of adjoint pairs. In \cite{Raf}, Rafael gave  conditions in terms of the unit and counit of an adjunction $(F,G)$ for the functors $F$ or $G$ to be separable. 

\smallskip
In this paper, we bring together two important refinements of the notion of separable functor that have appeared in the literature. In \cite{Ard1}, Ardizzoni and Menini introduced heavily separable functors, i.e., separable functors $F$ such that the splitting natural transformation $P$ is compatible with compositions in $\mathcal D$ in a certain manner (see 
also  Ardizzoni and Menini \cite{Ard0}, \cite{Ard2}). 
On the other hand,  suppose that we have functors $F:\mathcal C\longrightarrow \mathcal D$ and $I:\mathcal C\longrightarrow \mathcal X$. Then, the functor $F$ is said to be $I$-separable if  the natural transformation on morphism spaces induced by $F$ is split up to the natural transformation induced by $I$. This is known as separability of the second kind, which was introduced by 
Caenepeel and Militaru \cite{CM}. We combine these ideas to consider functors $F:\mathcal C\longrightarrow \mathcal D$ which are heavily $I$-separable, where 
$I$ is a functor $I:\mathcal C\longrightarrow \mathcal X$ (see Definition \ref{D2.1}). 

\smallskip
Let $(F:\mathcal C\longrightarrow \mathcal D, G:\mathcal D\longrightarrow \mathcal C)$ be an adjoint pair. Let $\eta: 1_{\mathcal C}\longrightarrow GF$ be the unit and
let   $\varepsilon: FG\longrightarrow 1_{\mathcal D}$ be the counit of this adjunction. We obtain a Rafael type theorem which shows that for any $I:\mathcal C\longrightarrow 
\mathcal X$, the left adjoint $F$ is heavily $I$-separable if and only if 
    there is a natural transformation $\gamma : IGF \longrightarrow I$ such that 
    \begin{equation} \gamma \circ I\eta = id \qquad 
\gamma \circ (\gamma GF) = \gamma \circ (IG\varepsilon F)
\end{equation} Similarly, for any $J:\mathcal D\longrightarrow \mathcal Y$, the right adjoint $G$  is heavily $J$-separable if and only if there is a natural transformation $\delta : J \longrightarrow JFG$ such that  
\begin{equation}J\varepsilon \circ \delta = id \qquad (\delta FG) \circ \delta = (JF\eta G) \circ \delta\end{equation}
We first apply our results to modules over preadditive categories. Following the philosophy of Mitchell \cite{Mit}, a small preadditive category $\mathscr R$ behaves like a ring with several objects. The category $\mathcal M_{\mathscr R}$ of right $\mathscr R$-modules consists of additive functors $\mathscr R^{op}\longrightarrow \mathbf{Ab}$ taking values in the category of abelian groups. If $\phi:\mathscr R\longrightarrow \mathscr S$ is an additive functor, there is an adjoint pair $(\phi^*,\phi_*)$ where $\phi^*:\mathcal M_{\mathscr R}\longrightarrow 
\mathcal M_{\mathscr S}$ is the extension of scalars, and $\phi_*:\mathcal M_{\mathscr S}\longrightarrow \mathcal M_{\mathscr R}$ is the restriction of scalars. We then consider morphisms of small preadditive categories 
 \begin{equation}\label{1/3r}\mathscr{Q} \xrightarrow{\psi} \mathscr{R} \xrightarrow{\phi} \mathscr{S} \xleftarrow{\xi} \mathscr{T}
\end{equation} and give if and only if conditions for $\phi^*$ to be heavily $\psi_*$-separable, as well as if and only if conditions for $\phi_*$ to be heavily $\xi_*$-separable. These extend the fact that for a morphism $\varphi:R\longrightarrow S$ of rings, the restriction of scalars $\varphi_*$ is separable if and only if $S$ is separable as an $R$-algebra in the classical sense 
(see \cite[Proposition 1.3]{NVV}). Further, the extension of scalars $\varphi^*$ is separable if and only if $\varphi$ splits as a morphism of $R$-bimodules (see \cite[Proposition 1.3]{NVV}). For the case of ordinary rings, separability of the second kind has been studied by Caenepeel and Militaru \cite[$\S$ 4]{CM} for the functors between module categories arising from \eqref{1/3r}. Our results also extend the study of heavily separable ring homomorphisms and heavy separability idempotents from Ardizzoni and Menini \cite{Ard1}. 

\smallskip
The next context we consider is that of monads and comonads. We know that if $(F:\mathcal C\longrightarrow \mathcal D, G:\mathcal D\longrightarrow \mathcal C)$  is an adjoint pair of functors, the composition $GF$ is canonically equipped with the structure of a monad $\mathbf T$ on $\mathcal C$. Dually, $FG$ can be equipped with the structure of a comonad 
$\mathbf S$ on $\mathcal D$. In fact, one can turn around this point of view. By fixing a category $\mathcal C$ and a monad $\mathbf T$ on $\mathcal C$, one can look at the family of $\mathbf T$-adjunctions, i.e.,
adjoint pairs $(F:\mathcal C\longrightarrow \mathcal D, G:\mathcal D\longrightarrow \mathcal C)$ whose associated monad is $\mathbf T$. If $I:\mathcal C\longrightarrow \mathcal X$ is any functor and $(F,G)$, $(F',G')$ are $\mathbf T$-adjunctions, we show that the left adjoint $F$ is heavily $I$-separable if and only if so is $F'$. This means that for a given monad $\mathbf T$, we can ask if the family of $\mathbf T$-adjunctions as a whole, is heavily $I$-separable. A dual result holds for comonads. These results are motivated by the work of Mesablishvili \cite{Mes} with $I$-separability and families of adjunctions associated to a given monad or comonad.  For more on separability conditions and how they relate to monads and comonads, we also refer the reader to Chen \cite{Chen}. 

\smallskip
Now suppose that $(L,R)$ is an adjunction such that the left adjoint $L$ can be equipped with the structure of a comonad $\mathbf L$ on $\mathcal C$. Then, we know (see \cite[$\S$ 2.6]{BBW}) that the right adjoint $R$ can be equipped with the structure of a monad $\mathbf R$. We show that the functor $F^{\mathbf L}$ taking objects of $\mathcal C$ to free $\mathbf L$-coalgebras is heavily separable if and only if so is the functor $F_{\mathbf R}$ taking objects of $\mathcal C$ to free $\mathbf R$-algebras. Combining with the results of 
Ardizzoni and Menini \cite{Ard1}, this has two applications. First, we show that for any coring $C$, the free $C$-contramodule functor is heavily separable if and only if $C$ has an invariant grouplike element. Secondly, a ring homomorphism $R\longrightarrow S$ is heavily separable if and only if the free contramodule functor for the Sweedler coring $S\otimes_RS$ is heavily separable. 

\smallskip
The final context that we study in this paper is that of entwined modules. This is one of the original contexts studied by Caenepeel and Militaru while introducing separability of the second kind in \cite{CM}. We recall that an entwining structure consists of an algebra $A$ and a coalgebra $C$ bound together by a morphism $\psi:C\otimes A\longrightarrow A\otimes C$ so that the datum 
$(A,C,\psi)$ behaves like a bialgebra. Entwining structures were introduced by Brzezi\'{n}ski and Majid \cite{Brz0} in order to study noncommutative principal bundles. The category $\mathcal M_A^C(\psi)$ of entwined modules, i.e., modules over the datum $(A,C,\psi)$ is a generalization of several concepts in the literature such as relative Hopf modules, Doi-Hopf modules and Yetter-Drinfeld modules (see, for intance, \cite{Brz1}, \cite{Brz2}, \cite{Brz3}, \cite{Can1}, \cite{Can2}, \cite{Cae3}, \cite{Cae4}, \cite{CanDe}, \cite{Can}, \cite{MT}). The structure of the category 
$\mathcal M_A^C(\psi)$ can be studied by means of two adjoint pairs $(U^C,F^C)$ and
$(F_A,U_A)$  
\begin{equation}\label{entrwad}
\mathcal M_A\xleftrightarrows[\text{$\qquad F^C\qquad$}]{\text{$\qquad U^C\qquad$}}\mathcal M(\psi)^{C}_A \qquad \mathcal M(\psi)^{C}_A\xleftrightarrows[\text{$\qquad U_A\qquad$}]{\text{$\qquad F_A\qquad$}}\mathcal M^C
\end{equation} where $U^C$ is the functor that forgets the $C$-coaction and $U_A$ is the functor that forgets the $A$-action. The separability of the functors in \eqref{entrwad} is a topic that has been widely studied in the literature (see, for instance, \cite{BBR}, \cite{Art}, \cite{uni}, \cite{Can1}, \cite{Can2}, \cite{Cae5}, \cite{MT}). In this paper, we develop if and only if conditions for the functor $U^C$ to be heavily $U_A$-separable, as well as for the functor $U_A$ to be heavily $U^C$-separable. The former case is related to morphisms $\theta:C\otimes C\longrightarrow A$ satisfying certain conditions, while the latter is related to certain kinds of morphisms $\zeta:C\longrightarrow A\otimes A$. 

\section{Heavily $I$-separable functors and a Rafael type theorem}

Throughout this paper, if $\mathcal C$, $\mathcal C'$ are categories, we denote by $[\mathcal C,\mathcal C']$ the category of functors from 
$\mathcal C$ to $\mathcal C'$. We begin by introducing the following definition.

\begin{defn}\label{D2.1}
Let $F : \mathcal{C} \longrightarrow \mathcal{D}$ and  $I:\mathcal C\longrightarrow \mathcal X$ be functors. We say that  $F:\mathcal C\longrightarrow \mathcal D$ is heavily $I$-separable if there exists a natural transormation $P:\mathcal D(F(\_\_),F(\_\_))\longrightarrow \mathcal X(I(\_\_),I(\_\_))$ in $[\mathcal C^{op}\times \mathcal C,Set]$  that satisfies the following  conditions :

\smallskip
(a) The following diagram commutes for any objects $a$, $b$ $\in \mathcal C$
\begin{equation}\label{e2.1}
\begin{tikzcd}[row sep=3em, column sep=2.5em]
   \mathcal{C}(a,b) \arrow{r}{F_{a,b}} \arrow[swap]{d}{I_{a,b}} & \mathcal{D}(F(a)_, F(b)) \arrow{ld}{P_{a,b}}\\
    \mathcal{X}(I(a) _, I(b))
   \end{tikzcd}
   \end{equation} 

\smallskip
(b)  $P$ is compatible with compositions, i.e., for any $a$, $b$, $c\in 
\mathcal C$, we have 

\begin{equation}\label{pre2.2}
   \begin{tikzcd}[row sep=huge, column sep=huge]
    \mathcal{D}(Fa, Fb) \times \mathcal{D}(Fb, Fc) \arrow{r}{P_{a, b} \times P_{b, c}} \arrow[swap]{d}{\circ} & \mathcal{X}(Ia, Ib) \times \mathcal{X}(Ib, Ic) \arrow{d}{\circ}\\
    \mathcal{D}(Fa, Fc) \arrow{r}{P_{a, c}} & \mathcal{X}(Ia, Ic)
   \end{tikzcd}
\end{equation}
\end{defn}

If $I = 1_{\mathcal{C}}$, then Definition \ref{D2.1} reduces to the notion of heavy separability due to Ardizzoni and Menini \cite{Ard1}.  In general, given composable functors $\mathcal C\overset{F}{\longrightarrow}\mathcal D\overset{G}{\longrightarrow}\mathcal E$, it may be easily verified that: (1) if $F$ is heavily $I$-separable and $G$ is heavily separable, then $GF$ is heavily $I$-separable, and (2) if $GF$ is heavily $I$-separable, then so is $F$.

\begin{lem}\label{L2.2f}
A heavily separable functor $G : \mathcal{C} \longrightarrow \mathcal{D}$ is heavily $I$-separable for any $I : \mathcal{C} \longrightarrow \mathcal{X}$. Additionally, if $I$ is fully faithful,  any heavily $I$-separable functor  $G : \mathcal{C} \longrightarrow \mathcal{D}$ is also heavily separable. 
\end{lem}

\begin{proof}
Since the identity functor $1_{\mathcal C}$ is heavily $I$-separable, it follows from the observations above that any heavily separable functor is also heavily $I$-separable.   The last statement follows from the fact that when $I$ is fully faithful, we have $ \mathcal X(I(\_\_),I(\_\_))\cong \mathcal C(\_\_,\_\_)$.
\end{proof}

\smallskip
Our main result in this section is a Rafael type theorem for heavily $I$-separable functors. 
\begin{Thm} \label{RTT}  Let $(F:\mathcal C\longrightarrow \mathcal D, G:\mathcal D\longrightarrow \mathcal C)$ be an adjoint pair, having unit $\eta: 1_{\mathcal C}\longrightarrow GF$ and
counit $\varepsilon: FG\longrightarrow 1_{\mathcal D}$. 
 Let $I : \mathcal{C} \longrightarrow \mathcal{X}$ and $J : \mathcal{D} \longrightarrow \mathcal{Y}$ be functors. Then, we have
 
 \smallskip
(1)  The left adjoint $F$ is heavily $I$-separable if and only if there is a natural transformation $\gamma : IGF \longrightarrow I$ such that $\gamma \circ I\eta = id$ and
\begin{equation}\gamma \circ (\gamma GF) = \gamma \circ (IG\varepsilon F)\end{equation}
(2) The right adjoint $G$ is heavily $J$-separable if and only if there is a natural transformation $\delta : J \longrightarrow JFG$ such that $J\varepsilon \circ \delta = id$ and
\begin{equation}(\delta FG) \circ \delta = (JF\eta G) \circ \delta\end{equation}
 
\end{Thm}

\begin{proof}
We only prove (1). Suppose first that there is a natural transformation $\gamma : IGF \longrightarrow I$ such that $\gamma \circ I\eta = id$ and $\gamma \circ (\gamma GF) = \gamma \circ (IG\varepsilon F)$. For any $(a, b) \in \mathcal{C}^{op} \times \mathcal{C}$, we set
\begin{equation}
    P_{a, b} : \mathcal{D}(Fa, Fb) \longrightarrow \mathcal{X}(Ia, Ib) \qquad h \mapsto \gamma_b \circ I(G(h) \circ \eta_a)
\end{equation}
It follows from the naturality of $\gamma$ and $\eta$ that $P : \mathcal{D}(F(\_\_) _, F(\_\_)) \longrightarrow \mathcal{X}(I(\_\_) _, I(\_\_))$ is a natural transformation. Further, for any $f \in \mathcal C(a,b)$, using the naturality of $\gamma$ and the fact that $\gamma \circ I\eta = id$, we have
\begin{equation}\label{pe2.1}
  \begin{split}
    &P_{a, b} \circ F_{a, b} (f) = P_{a, b}(F(f)) = \gamma_b \circ I(GF(f) \circ \eta_a) = \gamma_b \circ I(GF(f)) \circ I(\eta_a)\\
                                         &= I(f) \circ \gamma_a \circ I(\eta_a) = I_{a,b}(f)
  \end{split}
\end{equation}
Now, let $f \in \mathcal D(Fa,Fb)$ and $g\in \mathcal D(Fb, Fc)$. Then, we have
\begin{equation}\label{e2.y}
  \begin{split}
    P_{a, c}(g \circ f) &= \gamma_c \circ IG(g \circ f) \circ I(\eta_a)\\
                              &= \gamma_c \circ IG(g) \circ IG(\varepsilon_{Fb} \circ F(\eta_b)) \circ IG(f) \circ I(\eta_a)\qquad\left[\text{using } \varepsilon F \circ F\eta = 1_F\right]\\
                              &= \gamma_c \circ IG(g \circ \varepsilon_{Fb}) \circ IGF(\eta_b) \circ IG(f) \circ I(\eta_a)\\
                              &= \gamma_c \circ IG(\varepsilon_{Fc} \circ FG(g)) \circ IGF(\eta_b) \circ IG(f) \circ I(\eta_a)\qquad\left[\text{using naturality of }\varepsilon\right]\\
                              &= \gamma_c \circ \gamma_{GFc} \circ IGFG(g) \circ IGF(\eta_b) \circ IG(f) \circ I(\eta_a)\qquad\left[\text{using }\gamma \circ (\gamma GF) = \gamma \circ (IG\varepsilon F)\right]
  \end{split}
\end{equation}
Applying the naturality of $\gamma$ to the morphism $b \xrightarrow{G(g) \circ \eta_b} GFc$ in $\mathcal{C}$ gives the commutative square
\begin{equation}\label{ecd}
  \begin{tikzcd}[row sep=huge, column sep=huge]
    IGFb \arrow{r}{\gamma_b} \arrow[swap]{d}{IGFG(g) \circ IGF(\eta_b)} & Ib \arrow{d}{IG(g) \circ I(\eta_b)}\\
    IGFGFc \arrow{r}{\gamma_{GFc}} & IGFc
  \end{tikzcd}
\end{equation}
Using \eqref{ecd}, the expression in \eqref{e2.y} now becomes
\begin{equation}\label{ecde} P_{a, c}(g \circ f)=\gamma_c \circ IG(g) \circ I(\eta_b) \circ \gamma_b \circ IG(f) \circ I(\eta_a) = P_{b, c}(g) \circ P_{a, b}(f)
\end{equation}
From \eqref{pe2.1} and \eqref{ecde}, it follows that $F$ is heavily $I$-separable.

\smallskip
Conversely, let $F$ be heavily $I$-separable. Then, we have a natural transformation $P : \mathcal{D}(F(\_\_) _, F(\_\_)) \longrightarrow \mathcal{X}(I(\_\_) _, I(\_\_))$ satisfying the conditions in Definition \ref{D2.1}. 
We define  $\gamma : IGF \longrightarrow I$ by setting
\begin{equation} \gamma_a = P_{GFa, a}(\varepsilon_{Fa}) : IGFa \longrightarrow Ia
\end{equation} for each $a\in \mathcal{C}$. The naturality of $\gamma$ follows from that of $P$ and $\varepsilon$. 
Applying the naturality of $P$  to the morphism $(GFa, a) \xrightarrow{((\eta_a)^{op}, 1_a)} (a, a)$ in $\mathcal{C}^{op} \times \mathcal{C}$, we have the commutative diagram:
\begin{equation}\begin{tikzcd}[row sep=huge, column sep=huge]
  \mathcal{D}(FGFa, Fa) \arrow{r}{P_{GFa, a}} \arrow[swap]{d}{\mathcal{D}(F(\eta_a)^{op}, 1_{Fa})} & \mathcal{X}(IGFa, Ia) \arrow{d}{\mathcal{X}(I(\eta_a)^{op}, 1_{Ia})}\\
  \mathcal{D}(Fa, Fa) \arrow{r}{P_{a, a}} & \mathcal{X}(Ia, Ia)
\end{tikzcd}\end{equation}
and the  identity $\varepsilon F \circ F\eta = 1_F$, we have
\begin{equation}
    \gamma_a \circ (I\eta)_a = P_{GFa, a}(\varepsilon_{Fa}) \circ I(\eta_a) = P_{a, a}(\varepsilon_{Fa} \circ F(\eta_a)) = P_{a, a}(1_{Fa}) = P_{a, a}(F_{a, a}(1_a)) = I_{a, a}(1_a) = 1_{Ia}
\end{equation}
This shows that $\gamma \circ I\eta = id$. 
Now let $x, y\in \mathcal{C}$ and $f\in \mathcal D(Fx, Fy)$.
Applying the naturality of $P$ to the morphism $(GFy, y) \xrightarrow{((G(f) \circ \eta_x)^{op}, 1_y)} (x, y) $ in $\mathcal{C}^{op} \times \mathcal{C}$,  we obtain the commutative diagram
\begin{equation}\label{cd2.6}\begin{tikzcd}[row sep=large, column sep=large]
     \mathcal{D}(FGFy, Fy) \arrow{r}{P_{GFy, y}} \arrow[swap]{d}{\mathcal{D}((FG(f) \circ F(\eta_x))^{op}, 1_{Fy})} & \mathcal{X}(IGFy, Iy) \arrow{d}{\mathcal{X}((IG(f) \circ I(\eta_x))^{op}, 1_{Iy})}\\
     \mathcal{D}(Fx, Fy) \arrow{r}{P_{x, y}} & \mathcal{X}(Ix, Iy)
   \end{tikzcd}\end{equation}
Using \eqref{cd2.6}, the naturality of $\varepsilon$ and that $\varepsilon F \circ F\eta = id$, we get
\begin{equation}\label{pev4}
\begin{array}{ll}
     \gamma_y \circ IG(f) \circ (I\eta)_x &= P_{GFy, y}(\varepsilon_{Fy}) \circ IG(f) \circ I(\eta_x)\\
     & = P_{x, y}(\varepsilon_{Fy} \circ FG(f) \circ F(\eta_x)) = P_{x, y}(f \circ \varepsilon_{Fx} \circ F(\eta_x)) = P_{x, y}(f)\\
     \end{array}
\end{equation}
Finally, for any $a \in \mathcal{C}$, we have
\begin{equation}
  \begin{split}
    (\gamma \circ IG\varepsilon F)_a &= \gamma_a \circ (IG\varepsilon F)_a = \gamma_a \circ IG(\varepsilon_{Fa})\\
                                                   &= \gamma_a \circ IG(\varepsilon_{Fa}) \circ I(G(\varepsilon_{FGFa}) \circ \eta_{GFGFa})\qquad\quad\left[\text{using }G\varepsilon \circ \eta G = id\right]\\
                                                   &= \gamma_a \circ IG(\varepsilon_{Fa} \circ \varepsilon_{FGFa}) \circ I(\eta_{GFGFa})\\
                                                   &= P_{GFGFa, a}(\varepsilon_{Fa} \circ \varepsilon_{FGFa})\qquad\qquad\qquad\qquad\qquad\left[\text{using \eqref{pev4}}\right]\\
                                                   &= P_{GFa, a}(\varepsilon_{Fa}) \circ P_{GFGFa, GFa}(\varepsilon_{FGFa})\quad\qquad\qquad\left[\text{using condition \eqref{pre2.2}}\right]\\
                                                   &= \gamma_a \circ \gamma_{GFa} = (\gamma \circ \gamma GF)_a
  \end{split}
\end{equation}
so that $\gamma \circ (\gamma GF) = \gamma \circ (IG\varepsilon F)$. This proves the result.
\end{proof}
\vspace{18px}

\section{Functors between module categories}
 For a small preadditive category $\mathscr{R}$, let $\mathcal{M}_{\mathscr{R}}$ (respectively, ${}_{\mathscr{R}}\mathcal{M}$) denote the category of right (respectively, left) $\mathscr{R}$-modules, which consists of   additive functors $\mathscr{R}^{op} \longrightarrow \mathbf{Ab}$ (respectively, $\mathscr{R} \longrightarrow \mathbf{Ab}$) (see Mitchell \cite{Mit}). In particular, for every object $a \in \mathscr{R}$, $H_a := \mathscr R(\_\_, a)$ and $_{a}H := \mathscr {R}(a, \_\_)$ are right and left $\mathscr{R}$-modules respectively. Given a right $\mathscr{R}$-module $\mathscr M$ and a left $\mathscr{R}$-module $\mathscr N$, the tensor product $\mathscr M \otimes_\mathscr{R} N$ is defined (see for instance, \cite[$\S$ 2.2]{EV}) by the coend
\begin{equation}\label{coend}
    \mathscr M \otimes_\mathscr{R} \mathscr N := \int^{a \in \mathscr{R}} \mathscr{M}(a) \otimes_{\mathbb{Z}} \mathscr{N}(a) = \bigoplus_{a \in Ob(\mathscr{R})}\left( \mathscr{M}(a) \otimes_{\mathbb{Z}} \mathscr{N}(a)\right)/T
\end{equation}
where $T$ is the subgroup generated by elements of the form $\mathscr{M}(r)(x) \otimes y - x \otimes \mathscr{N}(r)(y), r \in \mathscr{R}(a,b), x \in \mathscr{M}(b), y \in \mathscr{N}(a)$. 

\smallskip Given two small preadditive categories $\mathscr{R}$ and $\mathscr{S}$, their tensor product $\mathscr{R} \otimes \mathscr{S}$ is the preadditive category with $Ob(\mathscr{R} \otimes \mathscr{S}) = Ob(\mathscr{R}) \times Ob(\mathscr{S})$ and ${\mathscr{R} \otimes \mathscr{S}}((a, b), (a', b')) = {\mathscr{R}}(a, a') \otimes_{\mathbb{Z}} {\mathscr{S}}(b, b')$ for any   $(a, b), (a', b') \in Ob(\mathscr{R}) \times Ob(\mathscr{S})$. Further, an $(\mathscr{R}, \mathscr{S})$-bimodule is simply a left $\mathscr{S}^{op} \otimes \mathscr{R}$-module. 
 For a functor $\phi : \mathscr{R} \longrightarrow \mathscr{S}$ between small preadditive categories, the restriction of scalars $\phi_{*} : \mathcal{M}_{\mathscr{S}} \longrightarrow \mathcal{M}_{\mathscr{R}}$ has a left  adjoint \begin{equation} \phi^{*} : \mathcal{M}_{\mathscr{R}} \longrightarrow \mathcal{M}_{\mathscr{S}}\qquad  \mathscr M \mapsto \mathscr M \otimes_{\mathscr{R}} \mathscr{S}(\_\_, \phi(\_\_))
 \end{equation} given by the extension of scalars.  The unit $\eta$ and counit $\varepsilon$ of this adjunction are given by 
 \begin{equation} \begin{array}{c}\eta = \left(\eta_{\mathscr M} = \left(\left(\eta_{\mathscr{M}}\right)_a : \mathscr{M}(a) \longrightarrow \mathscr{M} \otimes_\mathscr{R} \mathscr{S}(\phi(a), \phi(\_\_)), m \mapsto m \otimes 1_{\phi(a)}\right)_{a \in Ob(\mathscr{R})}\right)_{\mathscr M \in \mathcal{M}_{\mathscr{R}}} \\
 \varepsilon = \left(\varepsilon_{\mathscr N} = \left(\left(\varepsilon_{\mathscr N}\right)_b : (\mathscr{N} \circ \phi) \otimes_{\mathscr{R}} {\mathscr{S}}(b, \phi(\_\_)) \longrightarrow \mathscr{N}(b), n \otimes f \mapsto N(f)(n)\right)_{b \in \mathscr{S}}\right)_{\mathscr N \in \mathcal{M}_{\mathscr{S}}}\\
 \end{array}
 \end{equation}

\smallskip
We now fix morphisms of small preadditive categories 
\begin{equation}\mathscr{Q} \xrightarrow{\psi} \mathscr{R} \xrightarrow{\phi} \mathscr{S} \xleftarrow{\xi} \mathscr{T}
\end{equation} In this section, we will study conditions for the functors $\phi^{*}$ and $\phi_{*}$ to be heavily separable with respect to a restriction of scalars.
We first consider the heavy $\psi_{*}$-separability of $\phi^{*}$. We note that $\phi$ induces a morphism $\phi_{\psi(\_\_),\_\_} : {\mathscr{R}}(\psi(\_\_),\_\_) \longrightarrow {\mathscr{S}}(\phi(\psi(\_\_)),\phi(\_\_))$ of $(\mathscr{R}, \mathscr{Q})$-bimodules.
\begin{lem}\label{y1}
  There is a one one correspondence between natural transformations $\gamma : \psi_{*}\phi_{*}\phi^{*} \longrightarrow \psi_{*}$ such that $\gamma \circ \psi_{*}\eta = id$ and morphisms $\alpha : {\mathscr{S}}\left(\phi(\psi(\_\_)),\phi(\_\_)\right) \longrightarrow {\mathscr{R}}\left(\psi(\_\_),\_\_\right)$ of $(\mathscr{R}, \mathscr{Q})$-bimodules such that $\alpha \circ \phi_{\psi(\_\_),\_\_} = id$ given by,
\begin{equation}\label{z1}
  \begin{split}
     \gamma &\mapsto \alpha := \left(\alpha_{(a,b)} := \left(\gamma_{H_b}\right)_a : H_b \otimes_{\mathscr{R}} {\mathscr{S}}\left(\phi(\psi(a)), \phi(\_\_)\right) \cong {\mathscr{S}}\left(\phi(\psi(a)), \phi(b)\right) \longrightarrow {\mathscr{R}}\left(\psi(a), b\right)\right)_{(a,b) \in {\mathscr{Q}^{op} \otimes \mathscr{R}}}\\
     \alpha &\mapsto \gamma := \left(\gamma_{\mathscr{M}} := \left(\left(\gamma_{\mathscr{M}}\right)_a : \mathscr{M} \otimes_{\mathscr{R}} {\mathscr{S}}\left(\phi(\psi(a)), \phi(\_\_)\right) \longrightarrow \mathscr{M}(\psi(a)), m \otimes f \mapsto \mathscr{M}\left(\alpha(f)\right)(m)\right)_{a \in \mathscr{Q}}\right)_{\mathscr{M} \in \mathcal{M}_{\mathscr{R}}}
  \end{split}
\end{equation}
\end{lem}
\begin{proof}
   Let $\gamma : \psi_{*}\phi_{*}\phi^{*} \longrightarrow \psi_{*}$ be a natural transformation satisfying $\gamma \circ \psi_{*}\eta = id$.  We first verify that the corresponding  $\alpha$, as defined in \eqref{z1}, is a morphism of $(\mathscr{R}, \mathscr{Q})$-bimodules.  We consider  $f^{op} \otimes g \in {\mathscr{Q}^{op} \otimes \mathscr{R}}((a,b), (a',b'))$. The natural transformation $\gamma$ gives us the commutative square :

\begin{equation}\label{dg6.5}
   \begin{tikzcd}[row sep=huge, column sep=huge]
    {\mathscr{S}}(\phi(\psi(\_\_)), \phi(b)) \arrow{r}{H_{\phi(g)}(\phi \circ \psi)} \arrow[swap]{d}{\gamma_{H_b}} &{\mathscr{S}}(\phi(\psi(\_\_)), \phi(b')) \arrow{d}{\gamma_{H_{b'}}}\\
     {\mathscr{R}}(\psi(\_\_), b) \arrow{r}{H_g \psi} & {\mathscr{R}}(\psi(\_\_), b')
   \end{tikzcd}
\end{equation}
Evaluating \eqref{dg6.5} at  $a \in \mathscr{Q}^{op}$ and applying the naturality of $\gamma_{H_{b'}}$ to the morphism $f \in \mathscr Q(a',a)$, we have the following commutative diagram:
\begin{equation}\small
   \begin{tikzcd}[row sep=huge, column sep=huge]
    {\mathscr{S}}(\phi(\psi(a)), \phi(b)) \arrow{r}{ \phi(g) \circ \_\_} \arrow[swap]{d}{\left(\gamma_{H_b}\right)_a} & {\mathscr{S}}(\phi(\psi(a)), \phi(b')) \arrow[swap]{r}[swap]{\_\_ \circ \phi(\psi(f))} \arrow{d}{\left(\gamma_{H_{b'}}\right)_a} & {\mathscr{S}}(\phi(\psi(a')), \phi(b')) \arrow{d}{\left(\gamma_{H_{b'}}\right)_{a'}}\\
     {\mathscr{R}}(\psi(a), b) \arrow{r}{g \circ \_\_} & {\mathscr{R}}(\psi(a), b') \arrow{r}{\_\_ \circ \psi(f)} & {\mathscr{R}}(\psi(a'),b')
   \end{tikzcd}
\end{equation}
Therefore, for any $h \in {\mathscr{S}}\left(\phi(\psi(a)), \phi(b)\right)$,
$$g \circ \alpha_{(a,b)}(h) \circ \psi(f) = g \circ \left(\gamma_{H_b}\right)_a(h) \circ \psi(f) = \left(\gamma_{H_{b'}}\right)_{a'}(\phi(g) \circ h \circ \phi(\psi(f))) = \alpha_{(a',b')}(\phi(g) \circ h \circ \phi(\psi(f)))$$This shows that $\alpha$ is a morphism of $(\mathscr{R}, \mathscr{Q})$-bimodules. Also, since $\gamma \circ \psi_{*}\eta = id$, hence for any $a \in \mathscr{Q}, b \in \mathscr{R}$, and $r \in {\mathscr{R}}(\psi(a), b)$,
\begin{equation*}
     r = \left(\gamma_{H_b}\right)_a\left(\left(\left(\psi_{*}\eta\right)_{H_b}\right)_a(r)\right) = \alpha_{(a,b)}\left(\left(\eta_{H_b}\right)_{\psi(a)}(r)\right) = \alpha_{(a,b)}(\phi(r)) = \alpha_{(a,b)}\left(\left(\phi_{\psi(\_\_),\_\_}\right)_{(a,b)}(r)\right)
\end{equation*}
Therefore, $\alpha \circ \phi_{\psi(\_\_),\_\_} = id$.

\smallskip
Conversely, let $\alpha : {\mathscr{S}}\left(\phi(\psi(\_\_)),\phi(\_\_)\right) \longrightarrow {\mathscr{R}}\left(\psi(\_\_),\_\_\right)$ be a morphism of $(\mathscr{R}, \mathscr{Q})$-bimodules such that $\alpha \circ \phi_{\psi(\_\_), \_\_} = id$. We consider the corresponding $\gamma$ as defined in \eqref{z1}. We note that for $\mathscr M \in \mathcal{M}_{\mathscr{R}}, a \in \mathscr Q$, the well-definedness of $\left(\gamma_{\mathscr{M}}\right)_a$ follows from the universal property of the coend \eqref{coend}. It may be verified that $\gamma$ is a natural transformation. Also, for any $\mathscr{M} \in \mathcal{M}_{\mathscr{R}}, a \in \mathscr{Q}, m \in \mathscr{M}(\psi(a))$,
\begin{equation*}
  \begin{split}
    \left(\gamma_{\mathscr{M}}\right)_a\left(\left(\left(\psi_{*}\eta\right)_{\mathscr{M}}\right)_a(m)\right) &= \left(\gamma_{\mathscr{M}}\right)_a\left(\left(\eta_{\mathscr{M}}\right)_{\psi(a)}(m)\right)\\
    &= \left(\gamma_{\mathscr{M}}\right)_a\left(m \otimes 1_{\phi(\psi(a))}\right) = \mathscr{M}\left(\alpha_{(a, \psi(a))}\left(1_{\phi(\psi(a))}\right)\right)(m) = \mathscr{M}\left(1_{\psi(a)}\right)(m) = m
  \end{split}
\end{equation*}
Thus, $\gamma \circ \psi_{*}\eta = id$. It may also be verified that the two assignments $\gamma \mapsto \alpha$ and $\alpha \mapsto \gamma$ are mutual inverses.
\end{proof}
\begin{Thm}\label{ext}
$\phi^{*}$ is heavily $\psi_{*}$-separable if and only if there is a morphism $\alpha : {\mathscr{S}}\left(\phi(\psi(\_\_)),\phi(\_\_)\right) \longrightarrow {\mathscr{R}}\left(\psi(\_\_),\_\_\right)$ of $(\mathscr{R}, \mathscr{Q})$-bimodules such that 

\smallskip
(1) $\alpha \circ \phi_{\psi(\_\_),\_\_} = id$ and

\smallskip
(2) For every $a \in \mathscr{Q}$, $b,c \in \mathscr{R}$, and every $f \in {\mathscr{S}}\left(\phi(\psi(a)), \phi(b)\right), g \in {\mathscr{S}}\left(\phi(b), \phi(c)\right)$ $$\alpha_{a,c}(g \circ f) = \alpha_{a,c}\left(g \circ \phi\left(\alpha_{a,b}(f)\right)\right)$$
\end{Thm}
\begin{proof}
  By Theorem \ref{RTT}, $\phi^{*}$ is heavily $\psi_{*}$-separable if and only if there is a natural transformation $\gamma : \psi_{*}\phi_{*}\phi^{*} \longrightarrow \psi_{*}$  such that $\gamma \circ \psi_{*}\eta = id$ and $\gamma \circ \gamma \phi_{*}\phi^{*} = \gamma \circ \psi_{*}\phi_{*}\varepsilon\phi^{*}$. By Lemma \ref{y1}, there is a one one correspondence between natural transformations $\gamma : \psi_{*}\phi_{*}\phi^{*} \longrightarrow \psi_{*}$  such that $\gamma \circ \psi_{*}\eta = id$ and morphisms $\alpha : {\mathscr{S}}\left(\phi(\psi(\_\_)),\phi(\_\_)\right) \longrightarrow {\mathscr{R}}\left(\psi(\_\_),\_\_\right)$ of $(\mathscr{R}, \mathscr{Q})$-bimodules such that $\alpha \circ \phi_{\psi(\_\_),\_\_} = id$, given by \eqref{z1}.

\smallskip
Additionally, we see that, for $\mathscr M \in \mathcal{M}_{\mathscr{R}}, a \in \mathscr{Q}, b, c \in \mathscr{R}, f \in {\mathscr{S}}\left(\phi(\psi(a)), \phi(b)\right), g \in {\mathscr{S}}\left(\phi(b), \phi(c)\right), m \in \mathscr{M}(c)$
\begin{equation*}
  \begin{array}{ll}
    & \gamma \circ \gamma \phi_{*}\phi^{*} = \gamma \circ \psi_{*}\phi_{*}\varepsilon\phi^{*}\\
    \Leftrightarrow\text{ }& \left(\gamma_{\mathscr{M}}\right)_a \circ \left(\left(\gamma\phi_{*}\phi^{*}\right)_{\mathscr{M}}\right)_a = \left(\gamma_{\mathscr{M}}\right)_a \circ \left(\left(\psi_{*}\phi_{*}\varepsilon\phi^{*}\right)_{\mathscr{M}}\right)_a\\
    \Leftrightarrow\text{ }& \left(\gamma_{\mathscr{M}}\right)_a \circ \left(\gamma_{\mathscr M \otimes_{\mathscr{R}} {\mathscr{S}}\left(\phi(\_\_), \phi(\_\_)\right)}\right)_a = \left(\gamma_{\mathscr{M}}\right)_a \circ \left(\varepsilon_{\mathscr M \otimes_{\mathscr{R}} {\mathscr{S}}\left(\_\_, \phi(\_\_)\right)}\right)_{\phi(\psi(a))}\\
    \Leftrightarrow\text{ }& \left(\gamma_{\mathscr{M}}\right)_a\left(\left(\gamma_{\mathscr M \otimes_{\mathscr{R}} {\mathscr{S}}\left(\phi(\_\_), \phi(\_\_)\right)}\right)_a\left((m \otimes g) \otimes f\right)\right) = \left(\gamma_{\mathscr{M}}\right)_a\left(\left(\varepsilon_{\mathscr M \otimes_{\mathscr{R}} {\mathscr{S}}\left(\_\_, \phi(\_\_)\right)}\right)_{\phi(\psi(a))}\left((m \otimes g) \otimes f\right)\right)\\
    \Leftrightarrow\text{ }& \left(\gamma_{\mathscr{M}}\right)_a\left(\left(\left(\mathscr M \otimes_{\mathscr{R}} {\mathscr{S}}\left(\phi(\_\_), \phi(\_\_)\right)\right)\left(\alpha_{(a, b)}(f)\right)\right)\left(m \otimes g\right)\right) = \left(\gamma_{\mathscr{M}}\right)_a\left(\left(\left(\mathscr M \otimes_{\mathscr{R}} {\mathscr{S}}\left(\_\_, \phi(\_\_)\right)\right)(f)\right)\left(m \otimes g\right)\right)\\
    \Leftrightarrow\text{ }& \left(\gamma_{\mathscr{M}}\right)_a\left(m \otimes \left(g \circ \phi\left(\alpha_{(a,b)}(f)\right)\right)\right) = \left(\gamma_{\mathscr{M}}\right)_a\left(m \otimes (g \circ f)\right)\\
    \Leftrightarrow\text{ }& \mathscr{M}\left(\alpha_{(a,c)}\left(g \circ \phi\left(\alpha_{(a,b)}(f)\right)\right)\right)(m) = \mathscr{M}\left(\alpha_{(a,c)}(g \circ f)\right)(m)\\
    \Leftrightarrow\text{ }& \alpha_{(a,c)}\left(g \circ \phi\left(\alpha_{(a,b)}(f)\right)\right) = \alpha_{(a,c)}(g \circ f)
  \end{array}
\end{equation*}
The last line follows by taking $\mathscr{M} = H_c$ and $m = 1_c$. The result is now clear.
\end{proof}
Our next result deals with the heavy $\xi_{*}$-separability of the restriction of scalars $\phi_{*} : \mathcal{M}_{\mathscr{S}} \longrightarrow \mathcal{M}_{\mathscr{R}}$.
\begin{lem}\label{y2}
  There is a one one correspondence between natural transformations $\delta : \xi_{*} \longrightarrow \xi_{*}\phi^{*}\phi_{*}$ such that $\xi_{*}\varepsilon \circ \delta = id$ and collections $\Gamma = \left\{ \Gamma_a := \sum_{b \in Ob(\mathscr{R})} \left(\sum_i f_{i}^{ab} \otimes g_{i}^{ab}\right) \in {\mathscr{S}}(\phi(\_\_), \xi(a)) \otimes_{\mathscr{R}} {\mathscr{S}}(\xi(a), \phi(\_\_)) : a \in {\mathscr{T}}\right\}$ of elements given by finite sums such that
\begin{equation}\label{cod1}
    \sum_{b \in Ob(\mathscr{R})} \left(\sum_i \left(\xi(t) \circ f_{i}^{cb}\right) \otimes g_{i}^{cb}\right) = \sum_{b \in Ob(\mathscr{R})} \left(\sum_i f_{i}^{ab} \otimes \left(g_{i}^{ab} \circ \xi(t)\right)\right)\quad\text{for every }t \in {\mathscr{T}}(c, a), a,c \in Ob(\mathscr T)
\end{equation} and
\begin{equation}\label{cod2}
   \sum_{b \in Ob(\mathscr{R})} \left(\sum_i f_{i}^{ab} \circ g_{i}^{ab}\right) = 1_{\xi(a)}\quad\text{for all }a \in \mathscr{T}
\end{equation}
The one one correspondence is given by :
\begin{equation}\label{z2}
  \begin{array}{ll}
   \delta &\mapsto \Gamma := \left\{\Gamma_a:=\left(\delta_{H_{\xi(a)}}\right)_a(1_{\xi(a)}) : a \in \mathscr{T}\right\}\\
   \Gamma &\mapsto \delta := \left(\delta_{\mathscr{N}} := \left(\left(\delta_{\mathscr{N}}\right)_a : n \mapsto \sum_{b \in Ob(\mathscr{R})} \left(\sum_i \mathscr{N}\left(f_{i}^{ab}\right)(n) \otimes g_{i}^{ab}\right)\right)_{a \in \mathscr{T}}\right)_{\mathscr{N} \in \mathcal{M}_{\mathscr{S}}}
  \end{array}
\end{equation}  
\end{lem}
\begin{proof}
   Suppose that $\delta : \xi_{*} \longrightarrow \xi_{*}\phi^{*}\phi_{*}$ is a natural transformation such that $\xi_{*}\varepsilon \circ \delta = id$. We consider the corresponding collection $\Gamma := \left\{\Gamma_a:=\sum_{b \in Ob(\mathscr{R})} \left(\sum_i f_{i}^{ab} \otimes g_{i}^{ab}\right) = \left(\delta_{H_{\xi(a)}}\right)_a(1_{\xi(a)}) : a \in \mathscr{T}\right\}$, defined by \eqref{z2}. Let $t \in {\mathscr{T}}(c,a)$. Considering $H_{\xi(t)} : H_{\xi(c)} \longrightarrow H_{\xi(a)}$ in $\mathcal{M}_{\mathscr{S}}$, the natural transformation $\delta$ gives the commutative square:
\begin{equation}\label{cd6.10}
   \begin{tikzcd}[row sep=huge, column sep=huge]
    {\mathscr{S}}(\xi(\_\_), \xi(c)) \arrow{r}{H_{\xi(t)}\xi} \arrow[swap]{d}{\delta_{H_{\xi(c)}}} & {\mathscr{S}}(\xi(\_\_), \xi(a)) \arrow{d}{\delta_{H_{\xi(a)}}}\\
     {\mathscr{S}}(\phi(\_\_), \xi(c)) \otimes_{\mathscr{R}} {\mathscr{S}}(\xi(\_\_), \phi(\_\_)) \arrow{r}{ H_{\xi(t)}\phi\text{ }\otimes_{\mathscr{R}}\text{ }id} &  {\mathscr{S}}(\phi(\_\_), \xi(a)) \otimes_{\mathscr{R}} {\mathscr{S}}(\xi(\_\_), \phi(\_\_))
   \end{tikzcd}
\end{equation}
Evaluating \eqref{cd6.10} at $c \in {\mathscr{T}}$ and applying the naturality of $\delta_{H_{\xi(a)}}$ to the morphism $t^{op} \in {\mathscr{T}^{op}}(a, c)$, we have the following commutative diagram :
\begin{equation}\label{diag1}\small
   \begin{tikzcd}[row sep=huge, column sep=huge]
    {\mathscr{S}}(\xi(c), \xi(c)) \arrow{r}{\xi(t)\text{ }\circ\text{ }\_\_} \arrow[swap]{d}{\left(\delta_{H_{\xi(c)}}\right)_c} & {\mathscr{S}}(\xi(c), \xi(a)) \arrow{d}{\left(\delta_{H_{\xi(a)}}\right)_c} & {\mathscr{S}}(\xi(a), \xi(a)) \arrow{l}[swap]{\_\_\text{ }\circ\text{ }\xi(t)} \arrow[swap]{d}[swap]{\left(\delta_{H_{\xi(a)}}\right)_a}\\
    {\mathscr{S}}(\phi(\_\_), \xi(c)) \otimes_{\mathscr{R}} {\mathscr{S}}(\xi(c), \phi(\_\_)) \arrow{r}{H_{\xi(t)}\phi\text{ }\otimes_{\mathscr{R}}\text{ }id} & {\mathscr{S}}(\phi(\_\_), \xi(a)) \otimes_{\mathscr{R}} {\mathscr{S}}(\xi(c), \phi(\_\_)) & {\mathscr{S}}(\phi(\_\_), \xi(a)) \otimes_{\mathscr{R}} {\mathscr{S}}(\xi(a), \phi(\_\_)) \arrow{l}[swap]{id\text{ }\otimes_{\mathscr{R}}\text{ }_{\xi(t)}H\phi}
   \end{tikzcd}
\end{equation}
where $_{\xi(t)}H : { }_{\xi(a)}H\longrightarrow { }_{\xi(c)}H$ is the morphism induced by $t \in {\mathscr{T}}(c, a)$.
The two commutative squares in \eqref{diag1} give
\begin{equation*}
  \begin{split}
     \sum_{b \in Ob(\mathscr{R})} \left(\sum_i \left(\xi(t) \circ f_{i}^{cb}\right) \otimes g_{i}^{cb}\right) &= \left(H_{\xi(t)}\phi \otimes_{\mathscr{R}} id\right)\left(\left(\delta_{H_{\xi(c)}}\right)_c\left(1_{\xi(c)}\right)\right) = \left(\delta_{H_{\xi(a)}}\right)_c(\xi(t))\\
&= \left(id \otimes_{\mathscr{R}}{ }_{\xi(t)}H\phi\right)\left(\left(\delta_{H_{\xi(a)}}\right)_a\left(1_{\xi(a)}\right)\right) = \sum_{b \in Ob(\mathscr{R})} \left(\sum_i f_{i}^{ab} \otimes \left(g_{i}^{ab} \circ \xi(t)\right)\right)
  \end{split}
\end{equation*}
Hence, $\Gamma$ satisfies condition \eqref{cod1}. Further, since $\xi_{*}\varepsilon \circ \delta = id$, hence for any $a \in \mathscr{T}$,
\begin{equation*}
   \begin{split}
       1_{\xi(a)} &= \left(\left(\xi_{*}\varepsilon\right)_{H_{\xi(a)}}\right)_a\left(\left(\delta_{H_{\xi(a)}}\right)_a\left(1_{\xi(a)}\right)\right) = \left(\varepsilon_{H_{\xi(a)}}\right)_{\xi(a)}\left(\sum_{b \in Ob(\mathscr{R})} \left(\sum_i f_{i}^{ab} \otimes g_{i}^{ab}\right)\right)\\
       &= \sum_{b \in Ob(\mathscr{R})} \left(\sum_i H_{\xi(a)}\left(g_{i}^{ab}\right)\left(f_{i}^{ab}\right)\right) = \sum_{b \in Ob(\mathscr{R})} \left(\sum_i f_{i}^{ab} \circ g_{i}^{ab}\right)
   \end{split}
\end{equation*}
Thus, $\Gamma$ satisfies condition \eqref{cod2}.

\smallskip
Conversely, let $\Gamma = \left\{ \Gamma_a : a \in \mathscr{T}\right\}$ be a collection satisfying conditions \eqref{cod1} and \eqref{cod2}. It may be checked that the corresponding $\delta$, as defined in \eqref{z2}, is a natural transformation $\xi_{*} \longrightarrow \xi_{*}\phi^{*}\phi_{*}$. Also, for any $\mathscr{N} \in \mathcal{M}_{\mathscr{S}}, a \in \mathscr{T}, n \in \mathscr{N}(\xi(a))$,
\begin{equation*}
  \begin{split}
    \left(\left(\xi_{*}\varepsilon\right)_{\mathscr{N}}\right)_a\left(\left(\delta_{\mathscr{N}}\right)_a(n)\right) &= \left(\varepsilon_{\mathscr{N}}\right)_{\xi(a)}\left(\sum_{b \in Ob(\mathscr{R})} \left(\sum_i \mathscr{N}\left(f_{i}^{ab}\right)(n) \otimes g_{i}^{ab}\right)\right)\\
   &= \sum_{b \in Ob(\mathscr{R})} \left(\sum_i \mathscr{N}\left(g_{i}^{ab}\right)\left(\mathscr{N}\left(f_{i}^{ab}\right)(n)\right)\right) = \sum_{b \in Ob(\mathscr{R})} \left(\sum_i \mathscr{N}\left(f_{i}^{ab} \circ g_{i}^{ab}\right)(n)\right)\\
   &= \mathscr{N}\left(\sum_{b \in Ob(\mathscr{R})} \left(\sum_i f_{i}^{ab} \circ g_{i}^{ab}\right)\right)(n) = \mathscr{N}\left(1_{\xi(a)}\right)(n) = n
  \end{split}
\end{equation*}
Thus, $\xi_{*}\varepsilon \circ \delta = id$. It may be verified that the assignments $\delta \mapsto \Gamma$ and $\Gamma \mapsto \delta$ are mutual inverses.
\end{proof}
\begin{Thm}\label{res}
   $\phi_{*}$ is heavily $\xi_{*}$-separable if and only if there is a collection of elements given  by finite sums$$\Gamma = \left\{\Gamma_a := \sum_{b \in Ob(\mathscr{R})} \left(\sum_i f_{i}^{ab} \otimes g_{i}^{ab}\right) \in {\mathscr{S}}\left(\phi(\_\_), \xi(a)\right) \otimes_{\mathscr{R}} {\mathscr{S}}\left(\xi(a), \phi(\_\_)\right) : a \in {\mathscr{T}}\right\}$$which satisfies the following conditions :
\begin{equation}\label{cond1}
  \small \sum_{b \in Ob(\mathscr{R})} \left(\sum_i \left(\xi(t) \circ f_{i}^{cb}\right) \otimes g_{i}^{cb}\right) = \sum_{b \in Ob(\mathscr{R})} \left(\sum_i f_{i}^{ab} \otimes \left(g_{i}^{ab} \circ \xi(t)\right)\right)\quad\text{for every }t \in {\mathscr{T}}(c, a)
\end{equation}
\begin{equation}\label{cond2}
\small   \sum_{b \in Ob(\mathscr{R})} \left(\sum_i f_{i}^{ab} \circ g_{i}^{ab}\right) = 1_{\xi(a)}\quad\text{for all }a \in \mathscr{T}
\end{equation}
\begin{equation}\label{cond3}
\small    \sum_{d \in Ob(\mathscr{R})} \left(\sum_k \left(\sum_{b \in Ob(\mathscr{R})} \left(\sum_i f_{i}^{ab} \otimes \left(g_{i}^{ab} \circ f_{k}^{ad}\right)\right)\right)\otimes g_{k}^{ad}\right) = \sum_{b \in Ob(\mathscr{R})} \left(\sum_i \left(f_{i}^{ab} \otimes 1_{\phi(b)}\right) \otimes g_{i}^{ab}\right)\quad\text{for all }a \in \mathscr{T}
\end{equation}
\end{Thm}
\begin{proof}
   By Theorem \ref{RTT}, $\phi_{*}$ is heavily $\xi_{*}$-separable if and only if there is a natural transformation $\delta : \xi_{*} \longrightarrow \xi_{*}\phi^{*}\phi_{*}$ such that $\xi_{*}\varepsilon \circ \delta = id$ and $\delta\phi^{*}\phi_{*} \circ \delta = \xi_{*}\phi^{*}\eta\phi_{*} \circ \delta$. By Lemma \ref{y2}, there is a one one correspondence given by \eqref{z2}, between natural transformations $\delta : \xi_{*} \longrightarrow \xi_{*}\phi^{*}\phi_{*}$ such that $\xi_{*}\varepsilon \circ \delta = id$ and collections $\Gamma = \left\{\Gamma_a := \sum_{b \in Ob(\mathscr{R})} \left(\sum_i f_{i}^{ab} \otimes g_{i}^{ab}\right) \in {\mathscr{S}}(\phi(\_\_), \xi(a)) \otimes_{\mathscr{R}} {\mathscr{S}}(\xi(a), \phi(\_\_)) : a \in \mathscr{T}\right\}$ which satisfy conditions \eqref{cond1} and \eqref{cond2}.

\smallskip
Additionally, we see that for $\mathscr{N} \in \mathcal{M}_{\mathscr{S}}, a \in \mathscr{T}, n \in \mathscr{N}(\xi(a))$,
\begin{equation*}
  \begin{array}{ll}
     &\delta\phi^{*}\phi_{*} \circ \delta = \xi_{*}\phi^{*}\eta\phi_{*} \circ \delta\\
     \Leftrightarrow\text{ }& \left(\left(\left(\delta\phi^{*}\phi_{*}\right)_{\mathscr{N}}\right)_a \circ \left(\delta_{\mathscr{N}}\right)_a\right)(n) = \left(\left(\left(\xi_{*}\phi^{*}\eta\phi_{*}\right)_{\mathscr{N}}\right)_a \circ \left(\delta_{\mathscr{N}}\right)_a\right)(n)\\
     \Leftrightarrow\text{ }& \left(\left(\delta\phi^{*}\phi_{*}\right)_{\mathscr{N}}\right)_a\left(\sum_{b \in Ob(\mathscr{R})} \left(\sum_i \mathscr{N}\left(f_{i}^{ab}\right)(n) \otimes g_{i}^{ab}\right)\right) = \left(\left(\xi_{*}\phi^{*}\eta\phi_{*}\right)_{\mathscr{N}}\right)_a\left(\sum_{b \in Ob(\mathscr{R})} \left(\sum_i \mathscr{N}\left(f_{i}^{ab}\right)(n) \otimes g_{i}^{ab}\right)\right)\\
     \Leftrightarrow\text{ }& \left(\delta_{(\mathscr{N} \circ \phi) \otimes_{\mathscr{R}} {\mathscr{S}}\left(\_\_, \phi(\_\_)\right)}\right)_a\left(\sum_{b \in Ob(\mathscr{R})} \left(\sum_i \mathscr{N}\left(f_{i}^{ab}\right)(n) \otimes g_{i}^{ab}\right)\right)\\ &\text{ }\qquad= \left(\eta_{\mathscr{N} \circ \phi} \otimes_{\mathscr{R}} 1_{{\mathscr{S}}\left(\xi(a), \phi(\_\_)\right)}\right)\left(\sum_{b \in Ob(\mathscr{R})} \left(\sum_i \mathscr{N}\left(f_{i}^{ab}\right)(n) \otimes g_{i}^{ab}\right)\right)\\
     \Leftrightarrow\text{ }& \sum_{d \in Ob(\mathscr{R})} \left(\sum_k \left((\mathscr{N} \circ \phi) \otimes_{\mathscr{R}} {\mathscr{S}}\left(\_\_, \phi(\_\_)\right)\right)\left(f_{k}^{ad}\right)\left(\sum_{b \in Ob(\mathscr{R})} \left(\sum_i \mathscr{N}\left(f_{i}^{ab}\right)(n) \otimes g_{i}^{ab}\right)\right) \otimes g_{k}^{ad}\right)\\ &\text{ }\qquad= \sum_{b \in Ob(\mathscr{R})} \left(\sum_i \left(\mathscr{N}\left(f_{i}^{ab}\right)(n) \otimes 1_{\phi\left(b\right)}\right) \otimes g_{i}^{ab}\right)\\
     \Leftrightarrow\text{ }& \sum_{d \in Ob(\mathscr{R})} \left(\sum_k \left(\sum_{b \in Ob(\mathscr{R})} \left(\sum_i \mathscr{N}\left(f_{i}^{ab}\right)(n) \otimes \left(g_{i}^{ab} \circ f_{k}^{ad}\right)\right)\right) \otimes g_{k}^{ad}\right)\\ &\text{ }\qquad= \sum_{b \in Ob(\mathscr{R})} \left(\sum_i \left(\mathscr{N}\left(f_{i}^{ab}\right)(n) \otimes 1_{\phi\left(b\right)}\right) \otimes g_{i}^{ab}\right)\\
     \Leftrightarrow\text{ }& \sum_{d \in Ob(\mathscr{R})} \left(\sum_k \left(\sum_{b \in Ob(\mathscr{R})} \left(\sum_i f_{i}^{ab} \otimes \left(g_{i}^{ab} \circ f_{k}^{ad}\right)\right)\right)\otimes g_{k}^{ad}\right) = \sum_{b \in Ob(\mathscr{R})} \left(\sum_i \left(f_{i}^{ab} \otimes 1_{\phi\left(b\right)}\right) \otimes g_{i}^{ab}\right)
  \end{array}
\end{equation*}
The last line follows by taking $\mathscr{N} = H_{\xi(a)}$ and $n = 1_{\xi(a)}$. The result is now clear.
\end{proof}

We now assume that the preadditive categories $\mathcal{R}, \mathcal{S}, \mathcal{Q}$ and $\mathcal{T}$ each have a single object, i.e.,  they are rings $R$, $S$, $Q$ and $T$ respectively. It follows that the morphisms $Q \xrightarrow{\psi} R \xrightarrow{\phi} S \xleftarrow{\xi} T$ are ring homomorphisms. 
\begin{cor}\label{ind}
   $\phi^{*}$ is heavily $\psi_{*}$-separable if and only if there exists a morphism $\alpha: S \longrightarrow R$ of right $Q$-modules such that : (i) $\alpha \circ \phi = 1_{R}$ and (ii) $\alpha(s_1 \cdot \alpha(s_2)) = \alpha(s_1 s_2)$ for all $s_1, s_2 \in S$.
 \end{cor}
\begin{proof}
   The $(R, Q)$-bimodule map $\phi_{\psi(\_\_),\_\_} : {\mathscr{R}}(\psi(\_\_),\_\_) \longrightarrow {\mathscr{S}}(\phi(\psi(\_\_)),\phi(\_\_))$ is simply $\phi : R \longrightarrow S$. By Theorem \ref{ext},  $\phi^{*}$ is heavily $\psi_{*}$-separable if and only if there is a morphism $\alpha : S \longrightarrow R$ of $(R, Q)$-bimodules such that $\alpha \circ \phi = 1_{R}$ and $\alpha(s_1s_2) = \alpha(s_1\phi(\alpha(s_2))) = \alpha(s_1.\alpha(s_2))$ for all $s_1, s_2 \in S$.

\smallskip
Also, if $\alpha : S \longrightarrow R$ is a morphism of right $Q$-modules satisfying (i) and (ii), we see that
\begin{equation}\label{rlin}
\alpha(r\cdot s) =\alpha (\phi(r)s) = \alpha(\phi(r)\cdot \alpha(s)) = \alpha(\phi(r)\phi(\alpha(s))) = \alpha(\phi(r\alpha(s))) = r\alpha(s)
\end{equation} for all $r \in R, s \in S$, i.e., it is also left $R$-linear. The result is now clear.
\end{proof}
\begin{cor}\label{ideal}
  $\phi^{*}$ is heavily $\psi_{*}$-separable if and only if there exists a morphism $\alpha : S \longrightarrow R$ of right $Q$-modules such that $\alpha\circ \phi = 1_R$ and $ker(\alpha)$ is a left ideal of $S$.
\end{cor}
\begin{proof}
  If $\alpha : S \longrightarrow R$ satisfies (i) and (ii) in Corollary \ref{ind}, then for any $s_1 \in S, s_2 \in ker(\alpha)$,  $\alpha(s_1s_2) = \alpha(s_1\cdot
  \alpha(s_2)) = 0$ so that $ker(\alpha)$ is a left ideal of $S$.  Conversely, suppose that $\alpha : S \longrightarrow R$ is right $Q$-linear with $\alpha\circ \phi = 1_R$ and such that $ker(\alpha)$ is a left ideal of $S$. Then for any $s_1, s_2 \in S$, we note that
  \begin{equation}
  s_2 - \phi(\alpha(s_2))\in ker(\alpha)\Rightarrow s_1(s_2 - \phi(\alpha(s_2))) \in  ker(\alpha)\Rightarrow \alpha(s_1(s_2 - \phi(\alpha(s_2)))) = 0
  \end{equation} This proves the result.
\end{proof}
\begin{cor}\label{cor3.3}
    If $\psi : Q \longrightarrow R$ is an epimorphism of rings, then $\phi^{*}$ is heavily $\psi_{*}$-separable if and only if there is a ring homomorphism $\alpha : S \longrightarrow R$ satisfying $\alpha \circ \phi = 1_{R}$.
\end{cor}
\begin{proof} By \cite[Proposition 3.1]{Ard1}, $\phi^{*}$ is heavily separable if and only if it is a split monomorphism of rings. 
   Since $\psi$ is an epimorphism of rings, the restriction of scalars $\psi_{*}: \mathcal{M}_R \longrightarrow \mathcal{M}_Q$ is fully faithful (see \cite[Proposition XI.1.2]{Sten}). The result now follows 
   from Lemma \ref{L2.2f}. 
\end{proof}
\begin{rem} \end{rem}
\begin{enumerate}
\item We give an example to show that heavy separability of the second kind need not imply heavy separability. For this, we consider the ring map
$\phi : \mathbb{C} \longrightarrow M_2(\mathbb{R})$ given by $\phi(a+ib) = \begin{pmatrix}
     a & -b\\
     b & a
  \end{pmatrix}$ for $a+ib\in \mathbb C$. 
Since $M_2(\mathbb{R})$ is a simple ring, any unital ring homomorphism $M_2(\mathbb{R}) \longrightarrow \mathbb{C}$ would be injective, which is a contradiction since $\mathbb{C}$ is an integral domain while $M_2(\mathbb{R})$ is not. Thus, there are no unital ring homomorphisms of $M_2(\mathbb{R})$ into $\mathbb{C}$.  Using \cite[Proposition 3.1]{Ard1}, it follows that $\phi^*$ is not heavily separable.

\smallskip 
We now consider the inclusion $\mathbb{R} \xrightarrow{i} \mathbb{C}$. We note that the map $\alpha:M_2(\mathbb R)\longrightarrow \mathbb C$ given by $\alpha\begin{pmatrix}
             a & b\\
             c & d
  \end{pmatrix}=d-ib $ is right $\mathbb R$-linear and satisfies $\alpha\circ \phi=1_{\mathbb C}$. Further, the kernel of  $\alpha$ is $\left\{\begin{pmatrix}
            a & 0\\
            c & 0
 \end{pmatrix} : a,c \in \mathbb{R}\right\}$ which is a left ideal of $M_2(\mathbb R)$. It follows from Corollary \ref{ideal} that $\phi^*$ is heavily $i_*$-separable.

\item Let $R$ be a commutative ring and $S$ be an  $R$-algebra, given by the canonical ring map $\phi : R \longrightarrow S$ (so that $\phi(R) \subseteq Z(S)$). If $\phi^*$ is not heavily separable, then $\phi^*$ is not heavily $\psi_*$-separable for any ring map $\psi : Q \longrightarrow R$.  Otherwise  there exists a right $Q$-linear map $\alpha : S \longrightarrow R$ satisfying conditions (i) and (ii) of Corollary \ref{ind}. From \eqref{rlin}, we see that $\alpha$ is also left $R$-linear.  Using the fact that $R$ is commutative and $\phi(R)\subseteq Z(S)$, we have for any   $s_1, s_2 \in S$ 
\begin{equation*}\alpha(s_1s_2) = \alpha(s_1\cdot \alpha(s_2)) = \alpha(s_1\phi(\alpha(s_2))) = \alpha(\phi(\alpha(s_2))s_1) = \alpha(\alpha(s_2)\cdot s_1) = \alpha(s_2)\alpha(s_1) = \alpha(s_1)\alpha(s_2)\end{equation*} Then $\alpha$ is a ring homomorphism satisfying $\alpha \circ \phi = 1_{R}$. It follows that  $\phi^{*}$ is heavily separable, which is a contradiction. This applies in particular to the inclusion $i : \mathbb R \hookrightarrow \mathbb C$, since $i$ is not left invertible in the category of rings (as there are no unital ring maps from $\mathbb{C}$ to $\mathbb{R}$). 
\end{enumerate}

\smallskip
We finally consider the heavy $\xi_{*}$-separability of the restriction of scalars $\phi_{*}$ in the context of rings.
\begin{cor}
  $\phi_{*}$ is heavily $\xi_{*}$-separable if and only if there exists an element $\sum_i a_i \otimes b_i \in S\otimes_{R} S$ such that
  \begin{equation}\label{Eq1}
     \sum_i ta_i \otimes b_i = \sum_i a_i \otimes b_i t,\quad\text{for all }t \in T
  \end{equation}
  \begin{equation}\label{Eq2}
     \sum_i a_ib_i = 1
  \end{equation}
  \begin{equation}\label{Eq3}
     \sum_{i, k} a_i \otimes b_i a_k \otimes b_k = \sum_i a_i \otimes 1 \otimes b_i
  \end{equation}
\end{cor}
\begin{proof}
Since $T, R$ and $S$ are one-object categories, by Theorem \ref{res}, $\phi_{*}$ is heavily $\xi_{*}$-separable if and only if there is an element $\sum_i a_i \otimes b_i \in {S}\left(\phi(\_\_), \xi(a)\right) \otimes_{R} {S}\left(\xi(a), \phi(\_\_)\right) = S \otimes_{R} S$ satisfying conditions \eqref{Eq1}, \eqref{Eq2} and such that
$$\sum_{i, k} a_i \otimes b_ia_k \otimes b_k = \sum_k\left(\sum_i\left(a_i \otimes b_ia_k\right)\otimes b_k\right) = \sum_i a_i \otimes 1 \otimes b_i$$
The result is now clear.
\end{proof}

\section{Free functors associated to monads and comonads}
Suppose that $(F:\mathcal C\longrightarrow \mathcal D, G:\mathcal D\longrightarrow \mathcal C)$ is an adjunction with unit $\eta: 1_{\mathcal C}\longrightarrow GF$ and
counit $\varepsilon: FG\longrightarrow 1_{\mathcal D}$. We recall that the monad defined by the adjunction is the triple $\mathbf{T} = (GF, G\varepsilon F, \eta)$ where $G\varepsilon F : GFGF \longrightarrow GF$ is the multiplication and $\eta$ is the unit.  Dually, the adjunction gives a comonad $\mathbf{S} = (FG, F\eta G, \varepsilon)$ where $F\eta G : FG \longrightarrow FGFG$ is the comultiplication and $\varepsilon$ is the counit. Given a monad $\mathbf{T}$ on $\mathcal{C}$, any adjunction that determines the monad $\mathbf T$ is known as a  $\mathbf{T}$-adjunction (see \cite[$\S$ 3]{Mes}).  Similarly, one talks about $\mathbf{S}$-adjunctions where $\mathbf{S}$ is a comonad.  

\smallskip
\begin{thm}\label{C2.5}
Suppose that $(F:\mathcal C\longrightarrow \mathcal D, G:\mathcal D\longrightarrow \mathcal C)$ is an adjoint pair, having unit 
$\eta$ and counit $\varepsilon$. Let $I : \mathcal{C} \longrightarrow \mathcal{X}$ be a functor. Let $\mathbf{T}$ be the monad defined by the adjunction. Then, the following are equivalent :

\smallskip
(1) The functor $F$ is heavily $I$-separable

\smallskip
(2) For any $\mathbf{T}$-adjunction $(F':\mathcal C\longrightarrow \mathcal D', G':\mathcal D'\longrightarrow \mathcal C)$,
the functor $F'$ is heavily $I$-separable.
\end{thm}

\begin{proof}
We only need to show that (1) $\Rightarrow$ (2).  By Theorem \ref{RTT}, there is a natural transformation $\gamma : IGF \longrightarrow I$ such that
$$\gamma \circ I\eta =id\text{ and }\gamma \circ (\gamma GF) = \gamma \circ (IG\varepsilon F)$$
Let $\eta '$ be the unit and $\varepsilon '$ be the counit of the adjunction $(F', G')$. Since $(F',G')$ is a $\mathbf{T}$-adjunction,  the monad $(G'F',G'\varepsilon ' F',\eta')$ deteremined by the adjunction is $\mathbf{T} = (GF, G\varepsilon F,\eta)$. Thus, the natural transformation $\gamma : IG'F' = IGF \longrightarrow I$ satisfies
$$\gamma \circ I\eta ' = \gamma \circ I\eta = id\text{ and }\gamma \circ (\gamma G'F') = \gamma \circ (\gamma GF) = \gamma \circ (IG\varepsilon F) = \gamma \circ (IG'\varepsilon ' F')$$ Hence, $F'$ is heavily $I$-separable. 
\end{proof}

   Let  $\mathbf{T} = (T, \mu,\iota)$ be a monad on a category $\mathcal{C}$. We recall that the Eilenberg-Moore category of $\mathbf{T}$-algebras is the category $\mathcal{C}_{\mathbf{T}}$ whose objects are pairs $(x, h : Tx \longrightarrow x)$ such that $h \circ T(h) = h \circ \mu_x$ and $h \circ \iota_x = 1_x$.  The forgetful functor $U_\mathbf{T} : \mathcal{C}_{\mathbf{T}} \longrightarrow \mathcal{C}$ is right adjoint to the free functor $F_{\mathbf{T}} : \mathcal{C} \longrightarrow \mathcal{C}_{\mathbf{T}}, x \mapsto (Tx, \mu_x)$. We note that the monad defined by the adjunction $(F_{\mathbf{T}}, U_{\mathbf{T}})$ is $\mathbf{T}$. 
Similarly, for a comonad $\mathbf S$, there is an Eilenberg-Moore category $\mathcal C^{\mathbf S}$ of $\mathbf S$-coalgebras, with cofree coalgebra functor 
   $F^{\mathbf S}:\mathcal C\longrightarrow \mathcal C^{\mathbf S}$ and forgetful functor $U^{\mathbf S}: \mathcal C^{\mathbf S}\longrightarrow \mathcal C$, which gives an adjunction
   $(U^{\mathbf S},F^{\mathbf S})$.

\begin{lem}\label{L3.1} (\cite[Proposition 2.1.]{KeSt}) Let $(F_1:\mathcal C\longrightarrow \mathcal D, G_1:\mathcal D\longrightarrow \mathcal C)$
 (resp.  $(F_2:\mathcal C\longrightarrow \mathcal D, G_2:\mathcal D\longrightarrow \mathcal C)$) be an adjoint pair with  unit $\eta_1$ (resp. 
 $\eta_2$) and counit  $\varepsilon_1$ (resp. $\varepsilon_2$).  Then, there is a bijection
\begin{equation*}
     [\mathcal{C}, \mathcal{D}](F_1, F_2) \xrightarrow{\sim} [\mathcal{D}, \mathcal{C}](G_2, G_1) \qquad \alpha \mapsto \overline{\alpha} : = G_1\varepsilon_2 \circ G_1\alpha G_2 \circ \eta_1 G_2
\end{equation*} with inverse
\begin{equation*}
      [\mathcal{D}, \mathcal{C}](G_2, G_1) \xrightarrow{\sim} [\mathcal{C}, \mathcal{D}](F_1, F_2) \qquad \beta \mapsto \widehat{\beta} := \varepsilon_1 F_2 \circ F_1\beta F_2 \circ F_1\eta_2
\end{equation*}
Further, if $(F:\mathcal D\longrightarrow \mathcal E, G:\mathcal E\longrightarrow \mathcal D)$ is an adjoint pair, then for any natural transformation $\alpha : F_1 \longrightarrow F_2$, we have $\overline{F\alpha} = \overline{\alpha}G$. 
                           Similarly, if $(F':\mathcal B\longrightarrow \mathcal C,G':\mathcal C\longrightarrow \mathcal B)$ is an adjoint pair, then for any natural transformation $\alpha': F_1 \longrightarrow F_2$, we have
                           $\overline{\alpha' F'} = G'\overline{\alpha'}$.
\end{lem}

Given an adjunction $(L,R)$ with endofunctors $L$, $R$ on $\mathcal C$, we know from \cite[Section 3]{EMo} that $L$ can be equipped with a comonad structure $\mathbf L=(L,\Delta,\epsilon)$ if and only if $R$ can be equipped with a monad structure $\mathbf R=(R,\mu,\iota)$. In this setup, it has been shown in \cite[$\S$ 2.10]{BBW} that the free $\mathbf{L}$-coalgebra functor $F^{\mathbf{L}}$ is separable if and only if the free $\mathbf{R}$-algebra functor $F_{\mathbf{R}}$ is separable. We extend this result in the context of heavy separability of the second kind.

\smallskip
\begin{Thm}\label{T3.2}
Let $(L:\mathcal C\longrightarrow \mathcal C, R:\mathcal C\longrightarrow \mathcal C)$ be an adjoint pair of functors. Suppose that $\mathbf{L} := (L, \Delta, \epsilon)$ is a comonad with corresponding monad structure $\mathbf{R} := (R,\text{ }\mu := \overline{\Delta},\text{ }\iota := \overline{\epsilon})$ on $R$. Let $(I:\mathcal C\longrightarrow \mathcal C, J:\mathcal C\longrightarrow \mathcal C)$ be any adjoint pair such that
$I$ commutes with the comonad structure on $L$, i.e.
\begin{equation} IL = LI, \qquad I\Delta = \Delta I, \qquad I\epsilon = \epsilon I\end{equation}
Then,

\smallskip
(1) $F^{\mathbf{L}}$ is $I$-separable if and only if $F_{\mathbf{R}}$ is $J$-separable.

\smallskip
(2) If $I$ or $J$ is full, then $F^{\mathbf{L}}$ is heavily $I$-separable if and only if $F_{\mathbf{R}}$ is heavily $J$-separable.

\smallskip
(3) Let $(P : \mathcal{C} \longrightarrow \mathcal{A}, Q' : \mathcal{A} \longrightarrow \mathcal{C})$ be an adjunction with associated monad  $\mathbf{R}$ and let $(P' : \mathcal{B} \longrightarrow \mathcal{C}, Q : \mathcal{C} \longrightarrow \mathcal{B})$ be an adjunction with associated comonad $\mathbf{L}$. If $I$ or $J$ is full, then $Q$ is heavily $I$-separable if and only if $P$ is heavily $J$-separable. 
\end{Thm}
\begin{proof}
 Let $U^{\mathbf{L}} : \mathcal{C}^{\mathbf{L}} \longrightarrow \mathcal{C}$ and $U_{\mathbf{R}} : \mathcal{C}_{\mathbf{R}} \longrightarrow \mathcal{C}$ be the forgetful functors. Let $(\eta^{\mathbf{L}}, \varepsilon^{\mathbf{L}})$ and $(\eta_{\mathbf{R}}, \varepsilon_{\mathbf{R}})$ be the (unit, counit) of the adjunctions $(U^{\mathbf{L}}, F^{\mathbf{L}})$ and $(F_{\mathbf{R}}, U_{\mathbf{R}})$, respectively. The comonad determined by the adjunction $(U^{\mathbf{L}}, F^{\mathbf{L}})$ is $(U^{\mathbf{L}} F^{\mathbf{L}}, U^{\mathbf{L}}\eta^{\mathbf{L}} F^{\mathbf{L}}, \varepsilon^{\mathbf{L}}) = (L, \Delta, \epsilon)=\mathbf L$. The monad determined by the adjunction $(F_{\mathbf{R}}, U_{\mathbf{R}})$ is $(U_{\mathbf{R}} F_{\mathbf{R}}, U_{\mathbf{R}}\varepsilon_{\mathbf{R}} F_{\mathbf{R}}, \eta_{\mathbf{R}}) = (R, \mu, \iota)=\mathbf R$. 
 
 \smallskip
(1) Since $I$ commutes with $L$, $\Delta$ and $\epsilon$, it follows that $J$ commutes with $R$, $\mu$ and $\iota$. By Lemma \ref{L3.1}, there is a one-one correspondence between natural transformations $\delta: I \longrightarrow IU^{\mathbf{L}} F^{\mathbf{L}} = IL = LI$ such that $I\epsilon \circ \delta = I\varepsilon^{\mathbf{L}} \circ \delta = id$ and natural transformations $\gamma =\overline\delta : JR = JU_{\mathbf{R}} F_{\mathbf{R}} \longrightarrow J$ such that $\gamma \circ  \iota J=\gamma \circ J \iota = \gamma \circ J\eta_{\mathbf{R}} = id$.  The result now follows from \cite[Theorem 2.7]{CM}.   

\smallskip
(2) Suppose for instance that $I$ is full. By Theorem \ref{RTT}, $F^{\mathbf{L}}$ is heavily $I$-separable if and only if there exists $\delta : I \longrightarrow I U^{\mathbf{L}} F^{\mathbf{L}}$ such that $I\varepsilon^{\mathbf{L}} \circ \delta = id$ and $I U^{\mathbf{L}}\eta^{\mathbf{L}} F^{\mathbf{L}} \circ \delta = \delta U^{\mathbf{L}} F^{\mathbf{L}} \circ \delta$. Similarly, $F_{\mathbf{R}}$ is heavily $J$-separable if and only if there exists $\gamma : J U_{\mathbf{R}} F_{\mathbf{R}} \longrightarrow J$ such that $\gamma \circ J\eta_{\mathbf{R}} = id$ and $\gamma \circ \gamma U_{\mathbf{R}} F_{\mathbf{R}} = \gamma \circ J U_{\mathbf{R}} \varepsilon_{\mathbf{R}} F_{\mathbf{R}}$. We set $\gamma:=\overline\delta$. Proceeding as in  the proof of (1),   it remains to show that
\begin{equation}
   I U^{\mathbf{L}}\eta^{\mathbf{L}} F^{\mathbf{L}} \circ \delta =I\Delta \circ \delta = \delta L\circ \delta =\delta U^{\mathbf{L}} F^{\mathbf{L}} \circ \delta\qquad\Leftrightarrow \qquad \gamma \circ \gamma U_{\mathbf{R}} F_{\mathbf{R}} = \gamma\circ \gamma R=\gamma\circ J\mu =\gamma \circ J U_{\mathbf{R}} \varepsilon_{\mathbf{R}} F_{\mathbf{R}}
\end{equation}We first show that $L\delta \circ \delta = \delta L \circ \delta$ (this makes sense since $IL = LI$ by assumption). Let $a \in \mathcal{C}$. Since $I$ is full, we may 
choose $f\in \mathcal C(a,La)$ such that $I(f)=\delta_a:Ia\longrightarrow ILa$.   By the naturality of $\delta : I \longrightarrow IL = LI$ applied to the morphism $f : a \longrightarrow La$, it follows that $LI(f) \circ \delta_a = \delta_{La} \circ I(f)$, i.e.,  $L(\delta_a) \circ \delta_a = (\delta L)_a \circ \delta_a$. Since $a$ is an arbitrary object in $\mathcal{C}$, hence $L\delta \circ \delta = \delta L \circ \delta$.
\smallskip

Now, suppose that $I\Delta \circ \delta = \delta L\circ \delta $. Hence, $\Delta I\circ \delta= I\Delta \circ \delta= \delta L\circ \delta  = L \delta \circ \delta$. By Lemma \ref{L3.1}, we now have
\begin{equation*}
   \gamma \circ J \mu  = \overline{\delta} \circ \overline{\Delta I} = \overline{\Delta I \circ \delta} = \overline{L \delta \circ \delta} = \overline{\delta} \circ \overline{L \delta} = \gamma \circ \gamma R
\end{equation*}
 The converse follows similarly.
 
 \smallskip
(3) By Proposition \ref{C2.5}, $P$ is heavily $J$-separable if and only if $F_{\mathbf{R}}$ is heavily $J$-separable. By (2), the heavy $J$-separability of $F_{\mathbf{R}}$ is equivalent to the heavy $I$-separability of $F^{\mathbf{L}}$. By the dual of Proposition \ref{C2.5}, the result follows. 
\end{proof}
\begin{cor}\label{C3.3} If $L \dashv R : \mathcal{C} \longrightarrow \mathcal{C}$ is an adjoint pair with $\mathbf{L} = (L, \Delta, \epsilon)$ being a comonad and $\mathbf{R} = (R, \mu, \iota)$ being the corresponding monad, then

\smallskip
(1) $F^{\mathbf{L}}$ is heavily separable if and only if $F_{\mathbf{R}}$ is heavily separable.

\smallskip
(2) For any adjunction $P \dashv Q : \mathcal{C} \longrightarrow \mathcal{C}$ with associated comonad  $\mathbf{L} = (L, \Delta, \epsilon)$ and monad $\mathbf{R} = (R, \mu, \iota)$, $Q$ is heavily separable if and only if $P$ is heavily separable.
\end{cor}
\begin{proof}
   The result is clear by taking $I = J = 1_{\mathcal{C}}$ in Theorem \ref{T3.2}.
\end{proof}
\smallskip
We now consider an application. Let $A$ be a ring and let $\mathcal M_A$ denote the category of right $A$-modules. We recall (see, for instance, \cite[$\S$ 17]{Wis2003}) that an $A$-coring $C = (C, \Delta_C : C \longrightarrow C \otimes_A C, \epsilon_C : C \longrightarrow A)$ is a comonoid in the monoidal category of $(A,A)$-bimodules. We consider the adjoint pair
\begin{equation}\label{coim}
\mathcal M_A\xleftrightarrows[\text{$\qquad  Hom_A(C,\_\_)  \qquad$}]{\text{$\qquad  \_\_\otimes_AC \qquad$}}\mathcal M_A
\end{equation}
It is clear that the left adjoint $\_\_ \otimes_A C$ carries the canonical structure of a comonad $\mathbf L$ on $\mathcal M_A$, with comultiplication and counit induced by $\Delta_C$ and 
$\epsilon_C$ respectively. By Lemma \ref{L3.1}, its right adjoint $Hom_A(C _, \_\_)$ carries the structure of a monad $\mathbf R$ with multiplication and unit given by 
\begin{equation}
\begin{array}{c}
   Hom_A(C _, Hom_A(C _, \_\_)) \xrightarrow{\sim} Hom_A(C \otimes_A C _, \_\_) \xrightarrow{Hom_A(\Delta_{C},\_\_)} Hom_A(C _, \_\_) \\  1_{\mathcal{M}_A} \xrightarrow{\sim} Hom_A(A _, \_\_) \xrightarrow{Hom_A(\epsilon_{C},\_\_)} Hom_A(C _, \_\_)\\
   \end{array}
\end{equation}
The Eilenberg-Moore category $\mathcal{M}_A^{\mathbf{L}}$ of  $\mathbf L$ is the category $\mathcal{M}^C$ of right-$C$-comodules. The free $\mathbf{L}$-coalgebra functor $F^{\mathbf{L}} : \mathcal{M}_A \longrightarrow \mathcal{M}_A^{\mathbf{L}} = \mathcal{M}^C$ is the induction functor $M \mapsto (M \otimes_A C,\text{ }1_M \otimes_A \Delta_C)$. We also recall (see for instance, 
\cite[$\S$ 4.4]{BBW}) that an object of the Eilenberg-Moore category $\mathcal{M}_{A,\mathbf{R}}$ of $\mathbf{R}$-algebras is called a $C$-contramodule. The free $\mathbf{R}$-algebra functor $F_{\mathbf{R}} : \mathcal{M}_A \longrightarrow \mathcal{M}_{A,\mathbf{R}}, M \mapsto (Hom_A(C, M), Hom_A(\Delta_{C},M))$ is called the free $C$-contramodule functor.
\begin{thm}\label{P3.4}
The following statements are equivalent :\\
(1) The free $C$-contramodule functor $F_{\mathbf{R}}$ is heavily separable.\\
(2) The induction functor $F^{\mathbf{L}}$ is heavily separable.\\
(3) The coring $C$ has an invariant grouplike element, that is, an element $x \in C$ such that $ax = xa$ for all $a \in A$ and such that $\epsilon_C(x) = 1$ and $\Delta_C(x) = x \otimes x$.
\end{thm}
\begin{proof}
(1) $\Leftrightarrow$ (2)  By Corollary \ref{C3.3}, the result is clear.\\
(2) $\Leftrightarrow$ (3)  This is \cite[Theorem 2.8]{Ard1}.
\end{proof} 
\smallskip
Let $R \longrightarrow S$ be a ring map. We recall that  $S \otimes_R S$ is an $S$-coring (known as the Sweedler coring, see for instance, \cite[$\S$ 3]{Ard1}), with comultiplication $S \otimes_R S \longrightarrow (S \otimes_R S) \otimes_S (S \otimes_R S) \cong S \otimes_R S \otimes_R S$  given by $s_1 \otimes s_2 \mapsto s_1 \otimes 1 \otimes s_2$. The counit $S \otimes_R S \longrightarrow S$ is given by $s_1 \otimes s_2 \mapsto s_1s_2$.
\begin{cor}
   The following are equivalent for a ring homomorphism $R \longrightarrow S$ :\\
  (1) The restriction of scalars functor $\mathcal{M}_S \longrightarrow \mathcal{M}_R$ is heavily separable.\\
  (2) The free-$S \otimes_R S$-contramodule functor associated to the Sweedler coring $S \otimes_R S$ is heavily separable.
\end{cor}
\begin{proof}
   By \cite[$\S$ 3]{Ard1}, the restriction of scalars functor is heavily separable if and only if the Sweedler coring $S\otimes_RS$ has an invariant grouplike element. By Proposition \ref{P3.4}, the result is clear.
\end{proof}
\section{Functors associated to entwined modules}
Let $k$ be a commutative ring. In this section, all tensor products are taken over $k$. For a $k$-module $M$, we denote by $l_M$ and $r_M$, the canonical isomorphisms $l_M:k \otimes M \xrightarrow{\cong} M$ and $r_M:M \otimes k \xrightarrow{\cong} M$ respectively. We shall use the Sweedler notation for comultiplications and coactions throughout.

\smallskip
We recall (see for instance, \cite[$\S$ 1]{CM}) that a (right-right) entwining structure is a triple $(A, C, \psi)$ where $(A, \nabla_A : A \otimes A \longrightarrow A, i_A : k \longrightarrow A)$ is a $k$-algebra, $(C, \Delta_C : C \longrightarrow C \otimes C, \epsilon_C : C \longrightarrow k)$ is a $k$-coalgebra and $\psi : C \otimes A \longrightarrow A \otimes C, c \otimes a \mapsto a_{\psi} \otimes c^{\psi}$ (summation understood) is a $k$-linear map satisfying the following conditions :
\begin{equation}\label{ent1}
   \psi \circ (1_C \otimes \nabla_A) = (\nabla_A \otimes 1_C) \circ (1_A \otimes \psi) \circ (\psi \otimes 1_A)
\end{equation}
\begin{equation}\label{ent2}
  \psi \circ (1_C \otimes i_A) \circ {r_C}^{-1} = (i_A \otimes 1_C) \circ {l_C}^{-1}
\end{equation}
\begin{equation}\label{ent3}
  (1_A \otimes \Delta_C) \circ \psi = (\psi \otimes 1_C) \circ (1_C \otimes \psi) \circ (\Delta_C \otimes 1_A)
\end{equation}
\begin{equation}\label{ent4}
   r_A \circ (1_A \otimes \epsilon_C) \circ \psi = l_A \circ (\epsilon_C \otimes 1_A)
\end{equation}Further, an $(A, C, \psi)$-entwined module is a triple $(M, \rho_M, \rho^M)$ consisting of a $k$-module $M$, a right $A$-action $\rho_M : M \otimes A \longrightarrow M$ and a right $C$-coaction $\rho^M : M \longrightarrow M \otimes C$ on $M$ such that
\begin{equation}
  \rho^M \circ \rho_M = (\rho_M \otimes 1_C) \circ (1_M \otimes \psi) \circ (\rho^M \otimes 1_A)
\end{equation}
The category of $(A, C, \psi)$-entwined modules with $A$-linear, $C$-colinear maps is denoted by $\mathcal{M}(\psi)^{C}_{A}$. There are canonical adjunctions $(U^C,F^C)$ and
$(F_A,U_A)$ given by (see, for instance, \cite[$\S$ 1]{CM}):
\begin{equation}\label{entwad}
\mathcal M_A\xleftrightarrows[\text{$\qquad F^C\qquad$}]{\text{$\qquad U^C\qquad$}}\mathcal M(\psi)^{C}_A \qquad \mathcal M(\psi)^{C}_A\xleftrightarrows[\text{$\qquad U_A\qquad$}]{\text{$\qquad F_A\qquad$}}\mathcal M^C
\end{equation}
where $U^C$ and $U_A$ are the functors forgetting the $C$-coaction and the $A$-action respectively and
\begin{equation}
\begin{array}{c}
      F^C: \mathcal M_A\longrightarrow  \mathcal M(\psi)^{C}_A \qquad (M, \rho_M) \mapsto (M \otimes C,\text{ }\rho_{M \otimes C} = (\rho_M \otimes 1_C) \circ (1_M \otimes \psi),\text{ }\rho^{M \otimes C} = 1_M \otimes \Delta_C)\\ \\
      F_A: \mathcal M^C\longrightarrow  \mathcal M(\psi)^{C}_A \qquad (M, \rho^M) \mapsto (M \otimes A,\text{ }\rho_{M \otimes A} = 1_M \otimes \nabla_A,\text{ }\rho^{M \otimes A} = (1_M \otimes \psi) \circ (\rho^M \otimes 1_A))\\
      \end{array}
\end{equation}
The unit $\eta$ and counit $\varepsilon$ of the adjunction $(U^C, F^C)$ are given by :
\begin{equation}\label{uco1} \eta_{(M, \rho_M, \rho^M)} = \rho^M\text{ for all }(M, \rho_M, \rho^M) \in \mathcal{M}(\psi)^{C}_A \qquad \varepsilon_{(N, \rho_N)} = r_N \circ (1_N \otimes \epsilon_C)\text{ for all }(N, \rho_N) \in \mathcal{M}_A\end{equation}
The unit $\eta'$ and counit $\varepsilon'$ of the adjunction $(F_A, U_A)$ are given by :
\begin{equation}\label{uco2}\eta'_{(N, \rho^N)} = (1_N \otimes i_A) \circ r^{-1}_N\text{ for all }(N, \rho^N) \in \mathcal{M}^C \qquad \varepsilon'_{(M, \rho_M, \rho^M)} = \rho_M\text{ for all }(M, \rho_M, \rho^M) \in \mathcal{M}(\psi)^{C}_A\end{equation}

\smallskip
We now study equivalent conditions for the heavy $U_A$-separability of the functor $U^C$.
We begin by setting $\Sigma:= {Nat}(U_A F^C U^C F^C U^C, U_A)$. Let $\Omega$ be the set of all $k$-linear maps $T : C \otimes C \otimes C \longrightarrow A$ satisfying the following condition
\begin{equation}\label{D4.1}
\psi\circ (1_C\otimes T)\circ (\Delta_C\otimes 1_C\otimes 1_C)=(T\otimes 1_C)\circ (1_C\otimes 1_C\otimes \Delta_C):C\otimes C\otimes C\longrightarrow A\otimes C
\end{equation}

\begin{lem}\label{L4.2}
  The map $\alpha : \Sigma \longrightarrow \Omega$ given by,
\begin{equation}\label{E4.2}
   \alpha(\Phi) = T := r_A \circ (1_A \otimes \epsilon_C) \circ \Phi_{F^C\left(A, \nabla_A\right)} \circ (i_A \otimes 1_{C \otimes C \otimes C}) \circ l^{-1}_{C \otimes C \otimes C}
\end{equation}
is a bijection with inverse $\beta : \Omega \longrightarrow \Sigma$ given by,
\begin{equation}\label{E4.3}
   \beta(T) = \Phi := \left(\Phi_{(M, \rho_M, \rho^M)} := \rho_M \circ (1_M \otimes T) \circ (\rho^M \otimes 1_{C \otimes C})\right)_{(M, \rho_M, \rho^M) \in \mathcal{M}(\psi)^{C}_{A}}
\end{equation}
\end{lem}
\begin{proof}
  For $\Phi \in \Sigma$, let us show that $T = \alpha(\Phi)$ as defined in \eqref{E4.2} is an element of $\Omega$. We set $\underline{\Phi} := \Phi_{F^C\left(A, \nabla_A\right)}$ and $\overline{\Phi} := \Phi_{F_A\left(C, \Delta_C\right)}$. It follows from \eqref{ent1} and \eqref{ent3} that $\psi : F_A(C, \Delta_C) \longrightarrow F^C(A, \nabla_A)$ is a morphism in $\mathcal{M}(\psi)^{C}_A$. The naturality of $\Phi$ and \eqref{ent4} give the commutative diagram :
\begin{equation}\label{diagr}
\begin{tikzcd}[row sep=huge, column sep=huge]
  C \otimes A \otimes C \otimes C \arrow{r}{\overline{\Phi}} \arrow[swap]{d}{\psi \otimes 1_{C \otimes C}} & C \otimes A \arrow{r}{l_A \circ (\epsilon_C \otimes 1_A)} \arrow[swap]{d}{\psi} & A \arrow{d}{1_A}\\
  A \otimes C \otimes C \otimes C \arrow{r}{\underline{\Phi}} & A \otimes C \arrow{r}{r_A \circ (1_A \otimes \epsilon_C)} & A
\end{tikzcd}
\end{equation}
We set $\overline{\lambda} := l_A \circ (\epsilon_C \otimes 1_A) \circ \overline{\Phi}$ and $\underline{\lambda} := r_A \circ (1_A \otimes \epsilon_C) \circ \underline{\Phi}$. Then, for all $x, y, z \in C$,
\begin{equation}\label{lamb}
  \begin{split}
     T(x \otimes y \otimes z) &= \underline{\lambda}(1 \otimes x \otimes y \otimes z)\\
                                        &=  (r_A \circ (1_A \otimes \epsilon_C) \circ \underline{\Phi})(\psi(x \otimes 1) \otimes y \otimes z)\qquad[\text{using }\eqref{ent2}]\\
                                        &= \overline{\lambda}(x \otimes 1 \otimes y \otimes z)\qquad\quad\qquad\qquad\qquad\qquad[\text{using diagram }\eqref{diagr}]
  \end{split}
\end{equation}
We note that $\Delta_C$ induces natural left $C$-coactions on $U_A F^C U^C F^C U^C (C \otimes A), U_A(C \otimes A) \in \mathcal{M}^C$. Further, $\overline{\Phi} : C \otimes A \otimes C \otimes C \longrightarrow C \otimes A$ is left and right $C$-colinear. The left $C$-colinearity of $\overline{\Phi}$  gives us the commutative diagram
\begin{equation}\label{cd515r}\small
\begin{tikzcd}[row sep = 7em, column sep = 7em]
  C \otimes A \otimes C \otimes C \arrow{r}{\overline{\Phi}} \arrow[swap]{d}{\Delta_C \otimes 1_{A \otimes C \otimes C}} & C \otimes A \arrow{rd}{1_C \otimes 1_A} \arrow[swap]{d}{\Delta_C \otimes 1_A}\\
  C \otimes C \otimes A \otimes C \otimes C \arrow{r}{1_C \otimes \overline{\Phi}} & C \otimes C \otimes A \arrow{r}{(r_C \circ (1_C \otimes \epsilon_C)) \otimes 1_A} & C \otimes A
\end{tikzcd}
\end{equation}
The commutativity of \eqref{cd515r} implies that for all $x, y, z \in C$ 
\begin{equation}\label{Pt1}
  \begin{split}
     \overline{\Phi}(x \otimes 1 \otimes y \otimes z) &= ((r_C \circ (1_C \otimes \epsilon_C)) \otimes 1_A)(1_C \otimes \overline{\Phi})(\Delta_C \otimes 1_{A \otimes C \otimes C})(x \otimes 1 \otimes y \otimes z)\\
    &= (r_C \otimes 1_A)(x_{[1]} \otimes (\epsilon_C \otimes 1_A)(\overline{\Phi}(x_{[2]} \otimes 1 \otimes y \otimes z)))\\
    &= (1_C \otimes l_A)(x_{[1]} \otimes l_{A}^{-1}(\overline{\lambda}(x_{[2]} \otimes 1 \otimes y \otimes z)))\qquad\qquad[\text{using }r_C \otimes 1_A = 1_C \otimes l_A]\\
                                                                          &= x_{[1]} \otimes T(x_{[2]} \otimes y \otimes z)\qquad\qquad\qquad\qquad\qquad\qquad\qquad[\text{using }\eqref{lamb}]
  \end{split}
\end{equation}
Again, the right-$C$-colinearity of $\overline{\Phi}$, \eqref{ent3} and \eqref{ent4} give the commutative diagram

\begin{equation}\small \begin{tikzcd}[row sep = 7em, column sep = 7em]
  C \otimes A \otimes C \otimes C \arrow{r}{\overline{\Phi}} \arrow[swap]{d}{1_{C \otimes A \otimes C} \otimes \Delta_C} & C \otimes A \arrow{rd}{\psi} \arrow[swap]{d}{(1_C \otimes \psi) \circ (\Delta_C \otimes 1_A)}\\
  C \otimes A \otimes C \otimes C \otimes C \arrow{r}{\overline{\Phi} \otimes 1_C} & C \otimes A \otimes C \arrow{r}{(l_A \circ (\epsilon_C \otimes 1_A)) \otimes 1_C} & A \otimes C
\end{tikzcd}
\end{equation} 
Thus, for all $x, y, z \in C$,
\begin{equation}\label{Pt2}
  \begin{split}
   \psi(\overline{\Phi}(x \otimes 1 \otimes y \otimes z)) &= \left(\left(l_A \circ (\epsilon_C \otimes 1_A)\right) \otimes 1_C\right)(\overline{\Phi}(x \otimes 1 \otimes y \otimes z_{[1]}) \otimes z_{[2]})\\
                                                                                &= \overline{\lambda}(x \otimes 1 \otimes y \otimes z_{[1]}) \otimes z_{[2]}\\
                                                                                &= T(x \otimes y \otimes z_{[1]}) \otimes z_{[2]}\qquad\qquad\qquad[\text{using }\eqref{lamb}]
  \end{split}
\end{equation}
Combining \eqref{Pt1} and \eqref{Pt2}, we have for $x, y, z \in C$,
$$\psi(x_{[1]} \otimes T(x_{[2]} \otimes y \otimes z)) = \psi(\overline{\Phi}(x \otimes 1 \otimes y \otimes z)) = T(x \otimes y \otimes z_{[1]}) \otimes z_{[2]}$$ It follows that $T$ satisfies the condition in \eqref{D4.1}, i.e.,  $T \in \Omega$.

\smallskip
Conversely, given $T \in \Omega$, let $\Phi = \beta(T)$ as defined in \eqref{E4.3}. Using the commutativity of diagram \ref{D4.1}, we may check that for each $(M, \rho_M, \rho^M) \in \mathcal{M}(\psi)^{C}_{A}$, $\Phi_{(M, \rho_M, \rho^M)}$ is a morphism in $\mathcal M^C$. Further, for any morphism $f : (M, \rho_M, \rho^M) \longrightarrow (N, \rho_N, \rho^N)$ in $\mathcal{M}(\psi)^{C}_A$, we have for all $m \in M$ and $x, y \in C$,
\begin{equation*}
  \begin{split}
    \Phi_N(f(m) \otimes x \otimes y) &= \rho_N \circ (1_N \otimes T) \circ (\rho^N \otimes 1_{C \otimes C})(f(m) \otimes x \otimes y)\\
                                                    &= \rho_N \circ (1_N \otimes T)(f(m)_{[0]} \otimes f(m)_{[1]} \otimes x \otimes y)\\
                                                    &= \rho_N \circ (1_N \otimes T)(f(m_{[0]}) \otimes m_{[1]} \otimes x \otimes y)\qquad[\text{using the right-}C\text{-colinearity of }f]\\
                                                    &= f(m_{[0]})T(m_{[1]} \otimes x \otimes y)\\
                                                    &= f(m_{[0]}T(m_{[1]} \otimes x \otimes y))\qquad\qquad\qquad\qquad\quad[\text{using the right-}A\text{-linearity of }f]\\
                                                    &= f(\Phi_M(m \otimes x \otimes y))
  \end{split}
\end{equation*}so that $\Phi_N \circ (f \otimes 1_{C \otimes C}) = f \circ \Phi_M$. Thus, $\Phi \in \Sigma= {Nat}(U_A F^C U^C F^C U^C, U_A)$. It may be verified that
$\alpha$ and $\beta$ are inverses. 
\end{proof}

\begin{Thm}\label{T4.3}
The functor  $U^C$ is heavily $U_A$-separable if and only if there exists a $k$-linear map $\theta : C \otimes C \longrightarrow A$ such that for all $x, y, z \in C$
  \begin{equation}\label{E4.4}
    \theta(x \otimes y_{[1]}) \otimes y_{[2]} = \theta(x_{[2]} \otimes y)_{\psi} \otimes x^{\psi}_{[1]}
  \end{equation}
  \begin{equation}\label{E4.5}
    \theta \circ \Delta_C = i_A \circ \epsilon_C
  \end{equation}
  \begin{equation}\label{E4.6}
    \theta(y \otimes z_{[1]})_{\psi}.\theta(x^{\psi} \otimes z_{[2][1]}) \otimes z_{[2][2]} = \epsilon_C(y)\theta(x \otimes z_{[1]}) \otimes z_{[2]}
  \end{equation}
\end{Thm}
\smallskip
\begin{proof}
  By Theorem \ref{RTT}, $U^C$ is heavily $U_A$-separable if and only if there is a natural transformation $\gamma : U_A F^C U^C \longrightarrow U_A$ such that $\gamma \circ U_A\eta = id$ and $\gamma \circ (\gamma F^C U^C) = \gamma \circ (U_A F^C \varepsilon U^C)$. From the proof of \cite[Proposition 4.12]{CM}, there is a bijection  between the set $\Gamma$ of natural transformations $\gamma : U_A F^C U^C \longrightarrow U_A$ satisfying $\gamma \circ U_A(\eta) = id$ and the set $\Theta$ of $k$-linear maps $\theta : C \otimes C \longrightarrow A$ satisfying conditions \eqref{E4.4} and \eqref{E4.5} given by,
\begin{equation}\label{E4.7}
  \begin{split}
\gamma &\mapsto \theta := r_A \circ (1_A \otimes \epsilon_C) \circ \gamma_{F^C\left(A, \nabla_A\right)} \circ (i_A \otimes 1_{C \otimes C}) \circ l^{-1}_{C \otimes C}\\ \theta &\mapsto \gamma := \left(\gamma_{(M, \rho_M, \rho^M)} = \rho_M \circ (1_M \otimes \theta) \circ (\rho^M \otimes 1_C)\right)_{(M, \rho_M, \rho^M) \in \mathcal{M}(\psi)^{C}_{A}}
  \end{split}
\end{equation}
Thus, we only need to show that $\gamma\in \Gamma$ satisfies $\gamma \circ (\gamma F^C U^C) = \gamma \circ (U_A F^C \varepsilon U^C)$ if and only if the corresponding
 $\theta \in \Theta$ satisfies condition \eqref{E4.6}.

\smallskip
We fix $\gamma\in \Gamma$ and the corresponding element $\theta\in \Theta$. Let $\Phi_1 = \gamma \circ (\gamma F^C U^C)$ and $\Phi_2 = \gamma \circ (U_A F^C \varepsilon U^C)$. Since $\Phi_1, \Phi_2 \in \Sigma$, it follows by Lemma \ref{L4.2} that there are elements $T_1 = \alpha(\Phi_1), T_2 = \alpha(\Phi_2) \in \Omega$ where
\begin{equation}\label{E4.8}
  T_i = r_A \circ (1_A \otimes \epsilon_C) \circ (\Phi_i)_{F^{C}(A, \nabla_A)} \circ (i_A \otimes 1_{C \otimes C \otimes C}) \circ l^{-1}_{C \otimes C \otimes C},\qquad i = 1, 2
\end{equation}
Computing $\Phi_1$ at $F^C(A, \nabla_A) = A \otimes C \in \mathcal{M}(\psi)^{C}_{A}$ and using equation \eqref{E4.7} and condition \eqref{E4.4}, it follows that
\begin{equation*}
  \begin{split}
     &(\Phi_1)_{A \otimes C} = \gamma_{A \otimes C} \circ (\gamma F^C U^C)_{A \otimes C} = \gamma_{A \otimes C} \circ \gamma_{A \otimes C \otimes C}\\
     &= (\nabla_A \otimes 1_C) \circ (1_A \otimes \theta \otimes 1_C) \circ (1_{A \otimes C} \otimes \Delta_C) \circ (\nabla_A \otimes 1_{C \otimes C}) \circ (1_A \otimes \psi \otimes 1_C) \circ (1_{A \otimes C} \otimes \theta \otimes 1_C) \circ (1_{A \otimes C} \otimes 1_C \otimes \Delta_C)
  \end{split}
\end{equation*}
Similarly,
\begin{equation*}
\begin{array}{ll}
    (\Phi_2)_{A \otimes C} &= \gamma_{A \otimes C} \circ (U_A F^C \varepsilon U^C)_{A \otimes C} = \gamma_{A \otimes C} \circ (\varepsilon_{A \otimes C} \otimes 1_C)\\
    &= (\nabla_A \otimes 1_C) \circ (1_A \otimes \theta \otimes 1_C) \circ (1_{A \otimes C} \otimes \Delta_C) \circ (r_{A \otimes C} \otimes 1_C) \circ (1_{A \otimes C} \otimes \epsilon_C \otimes 1_C)\\
    \end{array}
\end{equation*}
A computation shows that for all $x, y, z \in C$,
\begin{equation}\label{CaT}
  \begin{split}
     \left[(\Phi_1)_{A \otimes C} \circ \left(i_A \otimes 1_{C \otimes C \otimes C}\right) \circ l^{-1}_{C \otimes C \otimes C}\right](x \otimes y \otimes z) &= \theta(y \otimes z_{[1]})_{\psi}.\theta(x^{\psi} \otimes z_{[2][1]}) \otimes z_{[2][2]}\\
     \left[(\Phi_2)_{A \otimes C} \circ \left(i_A \otimes 1_{C \otimes C \otimes C}\right) \circ l^{-1}_{C \otimes C \otimes C}\right](x \otimes y \otimes z) &= \epsilon_C(y)\theta(x \otimes z_{[1]}) \otimes z_{[2]}
  \end{split}
\end{equation}
which are the left-hand side and the right-hand side respectively, of equation \eqref{E4.6}.\\ \\
Now, suppose that $\gamma \circ (\gamma F^C U^C) = \gamma \circ (U_A F^C \varepsilon U^C)$. Then in particular, $(\Phi_1)_{A \otimes C} = (\Phi_2)_{A \otimes C}$, so that
\begin{equation*}
   \left[(\Phi_1)_{A \otimes C} \circ \left(i_A \otimes 1_{C \otimes C \otimes C}\right) \circ l^{-1}_{C \otimes C \otimes C}\right] = \left[(\Phi_2)_{A \otimes C} \circ \left(i_A \otimes 1_{C \otimes C \otimes C}\right) \circ l^{-1}_{C \otimes C \otimes C}\right]
\end{equation*}and hence, \eqref{E4.6} holds. Conversely, if $\theta$ satisfies condition \eqref{E4.6}, then by \eqref{CaT},
\begin{equation*}
  \begin{split}
    \left[(\Phi_1)_{A \otimes C} \circ \left(i_A \otimes 1_{C \otimes C \otimes C}\right) \circ l^{-1}_{C \otimes C \otimes C}\right] = \left[(\Phi_2)_{A \otimes C} \circ \left(i_A \otimes 1_{C \otimes C \otimes C}\right) \circ l^{-1}_{C \otimes C \otimes C}\right]
  \end{split}
\end{equation*}
so that $T_1 = T_2$, using \eqref{E4.8}. Thus, $\alpha(\Phi_1) = T_1 = T_2 = \alpha(\Phi_2)$, so that $\Phi_1 = \Phi_2$ since $\alpha$ is injective. Hence,
$\gamma \circ (\gamma F^C U^C) = \Phi_1 = \Phi_2 = \gamma \circ (U_A F^C \varepsilon U^C)$.
\end{proof}

\smallskip
We now come to the heavy $U^C$-separability of $U_A$. For a map $\zeta : C \longrightarrow A \otimes A$, we  write $\zeta(c) = \zeta^1(c) \otimes \zeta^2(c)$ (summation understood) for every $c \in C$. We begin by setting $\Pi := \text{Nat}(U^C, U^C F_A U_A F_A U_A)$. Let $\Lambda$ be the set of all $k$-linear maps $S : C \longrightarrow A \otimes A \otimes A$ satisfying the following condition
\begin{equation}\label{P4.5}
   (\nabla_A \otimes 1_A \otimes 1_A) \circ (1_A \otimes S) \circ \psi = (1_A \otimes 1_A \otimes \nabla_A) \circ (S \otimes 1_A) : C \otimes A \longrightarrow A \otimes A \otimes A
\end{equation}
\begin{lem}\label{L4.6}
   The map $\alpha' : \Pi \longrightarrow \Lambda$ given by
\begin{equation*}
   \alpha'(\Psi) = S := l_{A \otimes A \otimes A} \circ (\epsilon_C \otimes 1_{A \otimes A \otimes A}) \circ \Psi_{F_A(C, \Delta_C)} \circ (1_C \otimes i_A) \circ r^{-1}_{C}
\end{equation*}is a bijection with inverse $\beta' : \Lambda \longrightarrow \Pi$ given by
\begin{equation}
   \beta'(S) = \Psi := \left(\Psi_{(M, \rho_M, \rho^M)} := (\rho_M \otimes 1_{A \otimes A}) \circ (1_M \otimes S) \circ \rho^M\right)_{(M, \rho_M, \rho^M) \in \mathcal{M}(\psi)^{C}_A}
\end{equation}
\end{lem}
\begin{proof}
  The proof is similar to that of Lemma \ref{L4.2}.
\end{proof}
\smallskip
\begin{Thm}\label{T4.7}
  The functor $U_A$ is heavily $U^C$-separable if and only if there exists a $k$-linear map $\zeta : C \longrightarrow A \otimes A$ such that for all $c \in C, a \in A$
  \begin{equation}\label{E4.9}
    \zeta^1(c) \otimes \zeta^2(c)a = a_{\psi}\zeta^1(c^{\psi}) \otimes \zeta^2(c^{\psi})
  \end{equation}
  \begin{equation}\label{E4.10}
    \nabla_A \circ \zeta = i_A \circ \epsilon_C
  \end{equation}
  \begin{equation}\label{E4.11}
             \epsilon_C\left(c_{[1][1]}\right) \zeta^1\left(c_{[2]}\right)_{\psi} \otimes  \zeta^1\left((c_{[1][2]})^{\psi}\right) \otimes  \zeta^2\left((c_{[1][2]})^{\psi}\right). \zeta^2\left(c_{[2]}\right) =  \epsilon_C\left(c_{[1]}\right)\zeta^1\left(c_{[2]}\right) \otimes 1 \otimes  \zeta^2\left(c_{[2]}\right)
  \end{equation}
\end{Thm}
\begin{proof}
By Theorem \ref{RTT}, $U_A$ is heavily $U^C$-separable if and only if there is a natural transformation $\delta : U^C \longrightarrow U^C F_A U_A$ such that $U^C\varepsilon' \circ \delta = id$ and $(\delta F_A U_A) \circ \delta = (U^C F_A \eta' U_A) \circ \delta$. From the proof of \cite[Proposition 4.13]{CM}, there is a bijection between the set $\Xi_1$ of natural transformations $\delta : U^C \longrightarrow U^C F_A U_A$ such that $U^C\varepsilon' \circ \delta = id$ and the set $\Xi_2$ of $k$-linear maps $\zeta : C \longrightarrow A \otimes A$ satisfying conditions \eqref{E4.9} and \eqref{E4.10} given by,
\begin{equation}\label{E4.12}
  \begin{split}
      \delta &\mapsto \zeta := l_{A \otimes A} \circ (\epsilon_C \otimes 1_{A \otimes A}) \circ \delta_{C \otimes A} \circ (1_C \otimes i_A) \circ r^{-1}_{C}\\
      \zeta &\mapsto \delta := \left(\delta_{(M, \rho_M, \rho^M)} = (\rho_M \otimes 1_A) \circ (1_M \otimes \zeta) \circ \rho^M\right)_{(M, \rho_M, \rho^M) \in \mathcal{M}(\psi)^{C}_{A}}
  \end{split}
\end{equation}
Thus, it suffices to show that $\delta \in \Xi_1$ satisfies $(\delta F_A U_A) \circ \delta = (U^C F_A \eta' U_A) \circ \delta$ if and only if the corresponding $\zeta \in \Xi_2$ satisfies condition \eqref{E4.11}.

\smallskip
We fix $\delta \in \Xi_1$ and the corresponding $\zeta\in \Xi_2$.
Let $\Psi_1 = (\delta F_A U_A) \circ \delta$ and $\Psi_2 = (U^C F_A \eta' U_A) \circ \delta$. Since $\Psi_1, \Psi_2 \in \Pi$, it follows by Lemma \ref{L4.6} that there are elements $S_1 = \alpha'(\Psi_1), S_2 = \alpha'(\Psi_2) \in \Lambda$ where
\begin{equation}\label{E4.13}
   S_i =  l_{A \otimes A \otimes A} \circ (\epsilon_C \otimes 1_{A \otimes A \otimes A}) \circ (\Psi_i)_{F_A(C, \Delta_C)} \circ (1_C \otimes i_A) \circ r^{-1}_{C},\qquad i = 1, 2
\end{equation}
Computing $\Psi_1$ at $F_A(C) = C \otimes A \in \mathcal M(\psi)^{C}_{A}$,  using  \eqref{E4.12} and condition \eqref{E4.9}, it follows that
\begin{equation}\label{Ps1}
  \begin{split}
     &(\Psi_1)_{C \otimes A} = (\delta F_A U_A)_{C \otimes A} \circ \delta_{C \otimes A} = \delta_{C \otimes A \otimes A} \circ \delta_{C \otimes A}\\
     &= (1_{C \otimes A} \otimes 1_A \otimes \nabla_A) \circ (1_{C \otimes A} \otimes \zeta \otimes 1_A) \circ (1_C \otimes \psi \otimes 1_A) \circ (\Delta_C \otimes 1_{A \otimes A}) \circ (1_{C \otimes A} \otimes \nabla_A) \circ (1_C \otimes \zeta \otimes 1_A) \circ (\Delta_C \otimes 1_A)
  \end{split}
\end{equation}
Similarly,
\begin{equation}\label{Ps2}
  \begin{split}
    &(\Psi_2)_{C \otimes A} = (U^C F_A \eta' U_A)_{C \otimes A} \circ \delta_{C \otimes A} = (\eta'_{C \otimes A} \otimes 1_A) \circ \delta_{C \otimes A}\\
    &= (1_{C \otimes A} \otimes i_A \otimes 1_A) \circ (r^{-1}_{C \otimes A} \otimes 1_A) \circ (1_{C \otimes A} \otimes \nabla_A) \circ (1_C \otimes \zeta \otimes 1_A) \circ (\Delta_C \otimes 1_A)
  \end{split}
\end{equation}
A computation using \eqref{E4.13}, \eqref{Ps1}, \eqref{Ps2} shows that for all $c \in C$,
\begin{equation*}
  \begin{split}
      S_1(c) &= \epsilon_C\left(c_{[1][1]}\right)\zeta^1\left(c_{[2]}\right)_{\psi} \otimes \zeta^1\left((c_{[1][2]})^{\psi}\right) \otimes \zeta^2\left((c_{[1][2]})^{\psi}\right).\zeta^2\left(c_{[2]}\right)\\
      S_2(c) &= \epsilon_C\left(c_{[1]}\right)\zeta^1\left(c_{[2]}\right) \otimes 1 \otimes \zeta^2\left(c_{[2]}\right)
  \end{split}
\end{equation*}
which are the left-hand side and the right-hand side respectively, of equation \eqref{E4.11}.\\ \\
Now, suppose that $(\delta F_A U_A) \circ \delta = (U^C F_A \eta' U_A) \circ \delta$. Then in particular, $(\Psi_1)_{C \otimes A} = (\Psi_2)_{C \otimes A}$. Using \eqref{E4.13}, it follows that $S_1 = S_2$. Hence, \eqref{E4.11} holds. Conversely, if $\zeta$ satisfies condition \eqref{E4.11}, then $S_1 = S_2$. Thus, $\alpha'(\Psi_1) = S_1 = S_2 = \alpha'(\Psi_2)$, so that $\Psi_1 = \Psi_2$ since $\alpha'$ is injective. Hence, $(\delta F_A U_A) \circ \delta = \Psi_1 = \Psi_2 = (U^C F_A \eta' U_A) \circ \delta$.
\end{proof}

\end{document}